\documentclass[a4paper,12pt]{amsart}

\usepackage[utf8]{inputenc}
\usepackage[T1]{fontenc}

\usepackage[english]{babel}
\usepackage{amsmath, amsthm, amsfonts, amssymb}
\usepackage{graphicx}
\usepackage{hyperref}
\usepackage{color}

\usepackage[
   backend=biber,        
   sorting=nyt,          
   citestyle=authoryear, 
   bibstyle=alphabetic,  
]{biblatex}
\usepackage{enumerate}

\usepackage[margin=1in,top=1.2in, bottom=1.2in]{geometry}

\newtheorem{theorem}{Theorem}[section]
\newtheorem{lemma}[theorem]{Lemma}
\newtheorem{corollary}[theorem]{Corollary}
\newtheorem{proposition}[theorem]{Proposition}

\newtheorem{remark}[theorem]{Remark}

\newtheorem{definition}[theorem]{Definition}

\begin{document}

\title{Strong measurable continuous modifications of stochastic flows}

\author{Olivier Raimond}
\address{MODAL'X, Université Paris Nanterre, 200 Avenue de la
  République, 92000 Nanterre, France}
\email{oraimond@parisnanterre.fr}
\author{Georgii Riabov}
\address{Institute of Mathematics of NAS of Ukraine}
\email{ryabov.george@gmail.com}



\maketitle

\begin{abstract}
 This paper is devoted to stochastic flows of
measurable mappings in a locally compact separable metric space
$(M,\rho).$ We propose a new construction that produces strong measurable continuous modifications for certain stochastic flows of measurable mappings in metric graphs.
\end{abstract}

\tableofcontents{}

\section{Introduction}

This paper is devoted to the construction of stochastic flows of
measurable mappings in a locally compact separable metric space
$(M,\rho).$ We take the definition of stochastic flows of measurable
mappings from \cite{zbMATH02100692} (see the correction
\cite{zbMATH07226371} and the complete corrected version that can be
downloaded at
\href{https://arxiv.org/pdf/math/0203221.pdf}{arXiv:math/0203221v6};
in our references to \cite{zbMATH07226371} we use the numeration of
Theorems from
\href{https://arxiv.org/pdf/math/0203221.pdf}{arXiv:math/0203221v6}). To
describe the problem we briefly review the characterization of
stochastic flows of measurable mappings obtained in
\cite{zbMATH07226371}.  By $\mathcal{B}(M)$ we denote the Borel
$\sigma$-field in the space $M$. For every $n\in \mathbb{N},$ let
$(\mathsf{P}^{(n)}_t, t\geq 0)$ be a Feller transition function on
$M^n.$ We will say that $(\mathsf{P}^{(n)}_{\bullet}:n\in \mathbb{N})$ is a
\textit{consistent sequence of ``coalescing'' Feller transition
  functions on $M$}, if

\begin{enumerate}[(TF~1)]
\item\label{item:consistency} For any
  $\{i_1,\ldots,i_k\}\subset \{1,\ldots,n\},$ $t\geq 0,$ $x\in M^n$
  and $B\in \mathcal{B}(M^k)$
  $$
  \mathsf{P}^{(n)}_t(x,\pi^{-1}_{i_1,\ldots,i_k}B)=\mathsf{P}^{(k)}_t(\pi_{i_1,\ldots,i_k}x,B),
  $$
  where $\pi_{i_1.\ldots,i_k}:M^n\to M^k$ is defined by
  $\pi_{i_1,\ldots,i_k}x=(x_{i_1},\ldots,x_{i_k});$
  
\item\label{item:coalescing_property} For any $x\in M$ and $t\geq 0$
  $$
  \mathsf{P}^{(2)}_t((x,x),\Delta)=1,
  $$
  where $\Delta= \{(x,x)\,:\,x\in{M}\}$ is the diagonal in $M^2$.
\end{enumerate}

The results of \cite{zbMATH07226371} imply that to any consistent
sequence $(\mathsf{P}^{(n)}_{\bullet}:n\in \mathbb{N})$ of coalescing Feller
transition functions on $M$ one can associate a \textit{stochastic
  flow of measurable mappings
  $\psi=(\psi_{s,t}: -\infty<s\leq t<\infty)$ in $M$}
(see \cite[Def. 1.3.1]{zbMATH07226371}). Let $F$ be the space of all
measurable mappings $f:M\to M$ equipped with the cylindrical
$\sigma$-field. Denote by $C_0(M^n)$ the space of continuous functions
$f:M^n\to \mathbb{R}$ that vanish at infinity. A stochastic flow of
measurable mappings on $M$ is a family
$\psi=(\psi_{s,t}:-\infty<s\leq t<\infty)$ of random elements in $F$
defined on a probability space \((\Omega,\mathcal{A},\mathbb{P})\),
such that

\begin{enumerate}[(SF~1)]
\item\label{item:regularity}
  For any $s\leq t$, $n\in \mathbb{N}$, $x\in M^n$ and
  $f\in C_0(M^n),$
  $$
  \mathbb{E}f(\psi_{s,t}(x_1),\ldots,\psi_{s,t}(x_n))=\mathsf{P}^{(n)}_{t-s}f(x);
  $$

  %

\item \label{item:weak_flow} There exists a family
  $(\mathcal{J}_{t}: t\geq 0)$ of measurable mappings
  $\mathcal{J}_{t}:F\times M\to M,$ such that for all $s\leq t\leq u$
  and $x\in M,$ $\mathcal{J}_{t-s}(\psi_{s,t})(x)=\psi_{s,t}(x)$ a.s., 
  \begin{align*}
    \psi_{s,u}(x)=\mathcal{J}_{u-t}\left(\psi_{t,u}\right)\circ\psi_{s,t}(x) \mbox{ a.s.},
  \end{align*}
  and $\psi_{s,s}(x)=x$; 

\item \label{item:independence} For any
  $t_1\leq t_2\leq \ldots \leq t_n$, the family
  $(\psi_{t_i,t_{i+1}}:1\leq i\leq n-1)$ is a family of independent
  random mappings;
  

\item \label{item:sf_time_continuity} For any $f\in C_0(M)$ and $s\leq t$,
  $$
  \lim_{(u,v)\to
    (s,t)}\sup_{x\in{}M}\mathbb{E}\left[\big(f(\psi_{u,v}(x))-f(\psi_{s,t}(x))\big)^2\right]=0;
  $$
\item \label{item:sf_spatial_continuity} For any $f\in C_0(M), x\in M$
  and $s\leq t$,
  $$
  \lim_{y\to{}x}\mathbb{E}\left[\big(f(\psi_{s,t}(y))-f(\psi_{s,t}(x))\big)^2\right]=0
  \quad\mbox{ and }\quad
  \lim_{y\to\infty}\mathbb{E}\left[f(\psi_{s,t}(y))^2\right]=0.
  $$
\end{enumerate}

Let \(\psi\) be a stochastic flow of measurable mappings in \(M\) and
let \(\psi'\) be a family of random elements in \((F,\mathcal{F})\),
such that $\psi'_{s,s}(\omega,x)=x$ for all
$(s,x,\omega)\in \mathbb{R}\times M\times \Omega$, and
$\mathbb{P}\left[\psi'_{s,t}(x)=\psi_{s,t}(x)\right]=1$ for all
$s\leq t$ and $x\in M$, then $\psi'$ is also a stochastic flow of
measurable mappings in \(M\) (see \cite[Remark
1.3.2]{zbMATH07226371}). If \(\psi'\) is such that the mapping
$(s,t,x,\omega)\mapsto \psi'_{s,t}(\omega,x)$ is measurable, then the flow
$\psi'$ will be called a \textit{measurable} modification of $\psi$.

Note that for all \((s,x)\), the process $\psi_{s,\cdot}(x)$ is
Feller, and has a c\`adl\`ag modification. However, it is not obvious
that there exists a measurable modification $\psi'$ of $\psi$ with
c\`adl\`ag trajectories (i.e. such that \(\psi'_{s,\cdot}(x)\) is c\`adl\`ag
for all $(s,x,\omega)\in \mathbb{R}\times M \times \Omega$). We are
interested in the existence of such measurable modifications. It is
also natural to address the question of the existence of such
modification to satisfy the following \textit{strong flow property}
that improves the flow property (SF~\ref{item:weak_flow}):
\begin{enumerate}[(SF~1)]\setcounter{enumi}{5}
\item  \label{item:sfp}  For all $s\leq t\leq u$ and $\omega\in
  \Omega$ (to simplify the notation, we have omitted the dependency
  on \(\omega\)),
  $$
  \psi_{s,u} = \psi_{t,u}\circ\psi_{s,t}.
  $$
\end{enumerate}
In this paper we only work with flows with continuous trajectories
$t\mapsto \psi_{s,t}(x)$. In particular, we will always assume that
the transition function $(\mathsf{P}^{(1)}_t: t\geq 0)$ satisfies the
condition
\begin{enumerate}[(TF~1)]
  \setcounter{enumi}{2}
\item \label{item:continuity_of_trajectories_TF3} For any compact
  $K\subset{}M$ and $r>0$,\quad
  \[
    \lim_{t\downarrow{}0}t^{-1}\sup_{x\in{K}}\mathsf{P}^{(1)}_t\left(x,\{y:
      \rho(x,y)>r\}\right)=0.
  \]
\end{enumerate}

A measurable modification $\psi'$ of a stochastic flow of measurable
mappings $\psi$ on $M$ will be called a \textit{strong measurable
  continuous modification of \(\psi\)} if $\psi'$ has continuous
trajectories (i.e. \(\psi'_{s,\cdot}(x)\) is continuous for all
$(s,x,\omega)\in \mathbb{R}\times M \times \Omega$) and if \(\psi'\)
satisfies the strong flow property (SF~\ref{item:sfp}). A new method
to construct such modification is given in this paper. This method
requires an additional ``compactness'' condition (TF~\ref{item:14}) on
the sequence of \(n\)-point transition functions
$(\mathsf{P}^{(n)}_{\bullet}:n\in \mathbb{N})$ of \(\psi\). This condition will
be given in Section \ref{sec:2.3.3}.


The existence of strong measurable continuous modifications of
stochastic flows of measurable mappings is known in several cases. For
a stochastic flow of solutions to a stochastic differential equation
(SDE) with smooth coefficients in a finite-dimensional smooth manifold
$M$, the existence of a strong measurable continuous modifications was
proved in \cite{zbMATH01001278}. Strong measurable continuous
modifications of certain instantaneously coalescing stochastic flows
of measurable mappings in the real line were constructed in
\cite{zbMATH06929104}.  In \cite{zbMATH06357661} it was proved that
stochastic flows of kernels associated to the Brownian web possess
strong measurable continuous modifications. We note that in
\cite{zbMATH04024451} a strong stochastic flow associated to a
consistent sequences of coalescing transition functions on $M,$ even
without the Feller property, was constructed. However, this flow is
not measurable in either of the variables $s,t,x.$ We apply our
results to some stochastic flows in metric graphs. Stochastic flows of
solutions to SDE's on metric graphs were studied in
\cite{zbMATH06049114}, \cite{zbMATH06377556}, \cite{zbMATH06257629},
\cite{zbMATH06518035}.

The approach we propose is based on the analysis of families of
deterministic continuous mappings
$\theta_{s,\cdot}(x):[s,\infty)\to M,$ $\theta_{s,s}(x)=x.$ In
Section~\ref{sec:measurable_modification} we define the notion of
skeleton. A skeleton $\varphi$ is a sequence of continuous functions
$\varphi_n:[s_n,\infty)\to M,$ $n\in \mathbb{N}$, such that
\begin{itemize}
\item $\{(s_n,\varphi_n(s_n)):n\in \mathbb{N}\}$ is dense
  in $\mathbb{R}\times M$;
\item if $\varphi_n(t)=\varphi_m(t)$ then
  $\varphi_n(u)=\varphi_m(u)$ for all $u\geq t$;
\item for any compact $L\subset \mathbb{R}\times M$, the restrictions
  $\varphi_n$ to \([s,\infty)\) for all \(n,s\) such that
  $s_n\leq s$ and $(s,\varphi_n(s))\in L$ form a relatively compact
  set in the space of all continuous paths on \(M\).
\end{itemize}
Given a skeleton $\varphi$ we prove the existence of a family of
continuous mappings $\theta_{s,\cdot}(x):[s,\infty)\to M,$ such that
for all $(s,x)\in \mathbb{R}\times M,$ $\theta_{s,s}(x)=x$ and
$\theta_{s,t}(\varphi_n(s))=\varphi_n(t)$ for all $n\in \mathbb{N},$
$s_n\leq t.$ We refer to the latter property as the
\textit{preservation of the skeleton}. In our construction each
function $\theta_{s,\cdot}(x)$ is a limit point of the sequence
$(\theta_{s_n,\cdot}[s,\infty): s_n\leq s)$ in the space of continuous
paths. The existence of the family $\theta$ follows from the
measurable selection Lemma~\ref{lem:selection} given in
Appendix~\ref{sec:selection}.

The strong flow property 
\begin{equation}
\label{eq:sfp_deterministic}
\theta_{s,u}=\theta_{t,u}\circ \theta_{s,t}, \ \forall s\leq t\leq u,
\end{equation}
may fail due to the existence of the so-called bifurcation points of
the skeleton $\varphi$. The set $B(\varphi)$ of bifurcation points of
the skeleton $\varphi$ is defined in
Section~\ref{subsection:bifurcation_points}. The main idea of our
approach is to achieve \ref{eq:sfp_deterministic} by modifying
functions $\theta_{s,\cdot}(x)$ when they hit the set of bifurcation
points $B(\varphi)$ (to be precise, when they hit some larger closed
set $F\supset B(\varphi)$) in such a way that
\ref{eq:sfp_deterministic} holds. Theorem \ref{thm:sfp_deterministic}
gives a sufficient condition under which one can construct a new
family of continuous mappings $\psi_{s,\cdot}(x):[s,\infty)\to M$,
$\psi_{s,s}(x)=x$, that preserves the initial skeleton $\varphi$ and
that satisfies \ref{eq:sfp_deterministic}.

In Section~\ref{sec:strong-flow-meas}, we let \(\psi^{0}\) be a
stochastic flow of measurable mappings on $M$ that is associated to a
consistent sequence of coalescing Feller transition functions
$(\mathsf{P}^{(n)}_{\bullet}:n\in \mathbb{N})$ on $M$ that satisfies conditions
(TF~\ref{item:consistency}), (TF~\ref{item:coalescing_property}),
(TF~\ref{item:continuity_of_trajectories_TF3}) and
(TF~\ref{item:14}). A first measurable continuous modification
$\theta$ of $\psi^0$ is given. The results of
Section~\ref{sec:measurable_modification} are then applied to give
sufficient conditions under which there exists a strong measurable
continuous modification $\psi$ of $\psi^0$. These
conditions are stated in Theorem \ref{thm:sfp_stochastic}.

In Section~\ref{sec:icp} we study stochastic flows and their skeletons
that possess the instantaneous colasecing property. For such flows we
give sufficient conditions under which they possess strong measurable
continuous modifications (Proposition~ \ref{prop:icp-when-m}).

In Section~\ref{sec:sfp_IC} we study instantaneously coalescing
stochatic flows on metric graphs.
Corollary~\ref{cor:sfp_metric_graph} gives a sufficient condition for
the existence of strong measurable continuous modifications for
stochastic flows of measurable mappings in a metric graph.

In the final Section~\ref{sec:examples} we give several applications
of our approach. In Section~\ref{subsec:Coalescing independent Walsh
  Brownian motions on a metric graph} we prove the existence of a
strong measurable continuous stochastic flow of measurable mappings in a metric
graph, whose trajectories are coalescing Walsh Brownian motions
independent before the meeting time. In Section~\ref{subsec:Coalescing
  Tanaka flow} we prove the existence of a strong measurable
continuous stochastic flows of mappings in $\mathbb{R}$, whose
trajectories are solutions to Tanaka's SDE. In
Section~\ref{subsec:BK flow} we prove the existence of a strong
measurable continuous stochastic flow of mappings in $\mathbb{R},$
whose trajectories are solutions to the Harrison-Shepp SDE for the
skew Brownian motion (a Burdzy-Kaspi flow, see \cite{zbMATH02148685}). In
Section~\ref{subsec:Tanaka_graph} we prove the existence of a strong
measurable continuous stochastic flow of mappings in a metric graph $M,$
whose trajectories are solutions to the Tanaka's SDE on $M$ (see
\cite{zbMATH06049114}, \cite{zbMATH06257629}).

\section{Flow extensions of skeletons on \(M\)}
\label{sec:measurable_modification}

In Section~\ref{sec:notation_skeleton} we introduce \(X\), the
\textit{space of continuous paths in $M$}, and the notion of a
skeleton in $M$. In Section \ref{subsection:skeleton_extension} a
measurable mapping
\(\Theta:\mathbb{R}\times{M}\times\mathcal{S}(M)\to{}X\) that will
allow us to construct for every skeleton $\varphi$ a family of
continuous mappings $\theta_{s,\cdot}(x):[s,\infty)\to M$, such that
\begin{enumerate}[(i)]
\item for all $(s,x)\in \mathbb{R}\times M$, $\theta_{s,s}(x)=x$;\label{item:19}
\item for all \(n\in\mathbb{N}\) and \(t\geq s\geq s_{n}\),\label{item:20}
  $\theta_{s,t}(\varphi_n(s))=\varphi_n(t)$.
\end{enumerate}
A family \(\theta\) satisfying~\eqref{item:19} and~\eqref{item:20} will be
said to \textit{preserves the skeleton} \(\varphi\).

In Section \ref{subsection:bifurcation_points} the set $B(\varphi)$ of
bifurcation points of a skeleton $\varphi$ is defined. Out of $\theta$
we define a new family of mappings $\psi_{s,\cdot}(x):[s,\infty)\to M$
that by construction satisfies the property
$\psi_{s,t}\circ \psi_{r,s}(x)=\psi_{r,t}(x)$ at certain points
$(s,\psi_{r,s}(x))\in B(\varphi).$ In
Theorem~\ref{thm:sfp_deterministic} we give a sufficient condition
under which $\psi$ is a family of continuous mappings that preserves
the skeleton $\varphi$ and satisfies the strong flow property
\eqref{eq:sfp_deterministic}.

\subsection{The space of skeletons in \(M \)}
\label{sec:notation_skeleton}

As above, \((M,\rho)\) is a locally compact separable metric
space. Without loss of generality we will assume that $\rho$ is
complete and that all bounded subsets of $(M,\rho)$ are relatively
compact\footnote{When it is not the case, then there is a distance
  \(\rho'\) topologically equivalent to \(\rho\) for which these two
  additional properties are satisfied.}. For any interval
$I\subset \mathbb{R}$ (possibly, unbounded) the space \(C(I:M)\) of
continuous functions from \(I\) to \(M\) is equipped with the topology
of uniform convergence on compact subsets of $I $. Let us introduce
the space of continuous paths in $M$:
$$
X=\bigcup_{s\in \mathbb{R}}C([s,\infty):M).
$$
Each element of $X$ is a continuous function $f:[s,\infty)\to M$. The
starting time $s$ will be denoted by $i(f):$ $i(f)=s$.  If $f\in X$
and $s=i(f)$, we denote  $e(f)\in{}C(\mathbb{R}:M)$ the continuous
extension of $f$ onto $\mathbb{R},$
$$
e(f)(t)=f(s\vee t),\quad t\in \mathbb{R}.
$$ 
Let \(\delta\) be a distance on \(C(\mathbb{R}:M)\) associated with
the topology of uniform convergence on compact subsets of
$\mathbb{R}$\footnote{For example, one can take
  \(\delta(f,g)=\sum_{n\geq{1}}2^{-n}\left(1\wedge
    \sup_{x\in[-n,n]}\rho(f(x),g(x))\right)\).}. The space
\(\mathbb{R}\times{}C(\mathbb{R}:M)\) equipped with the metric
\(d_{1}((s,f),(t,g))=|t-s|+\delta(f,g)\) is a complete separable
metric space. Let
\(X'=\{(s,f)\in\mathbb{R}\times{}C(\mathbb{R}:M):\,\forall{}t\leq{s},f(t)=f(s)\}\).
Then \(X'\) is a closed subset of
\(\mathbb{R}\times{}C(\mathbb{R}:M)\) and is a complete separable
metric space.

Let us equip \(X\) with the distance \(d\) defined by
\(d(f,g)=|i(f)-i(g)|+\delta(e(f),e(g))\). Then the mapping
\(f\mapsto(i(f),e(f))\) from \(X\) onto \(X'\) is an isometry. As a
consequence, \((X,d)\) is a complete separable metric space and its
Borel \(\sigma\)-field \(\mathcal{B}(X)\) is the \(\sigma\)-field
generated by the mappings \(f\mapsto{}e(f)(t)\) for all
\(t\in\mathbb{R}\) and \(f\mapsto{}i(f)\).

For any \(f\in{X}\) and any interval \(I\), we will denote by \(fI\)
the restriction of \(e(f)\) to \(I\). Note that the mapping
\(f\mapsto{}fI\) is a continuous (hence, Borel measurable) mapping
from \(X\) to \(C(I:M)\).

Let \(\varphi=(\varphi_n:n\in \mathbb{N})\in X^{ \mathbb{N}}\) be a
sequence in $X$. For \(n\in \mathbb{N}\), set $s_n=i(\varphi{}_n)$ and
\(x_n=\varphi_n(s_n)\). For $s\in\mathbb{R}$, set
$I^s:=\{n\in\mathbb{N}:\ s_n\le s\}$.

\begin{definition}
  \label{def:skeleton}
  The sequence \(\varphi\in X^{\mathbb{N}}\) is called a skeleton if
  the following properties are satisfied:
  \begin{enumerate}[(Sk~1)]
  \item \label{item:coalescence} If \(s\in\mathbb{R}\) and if
    \((m,n)\in{I^{s}}\times{I^{s}}\) are such that
    \(\varphi_{m}(s)=\varphi_{n}(s)\), then
    \(\varphi_{m}[s,\infty)=\varphi_{n}[s,\infty)\)
    (i.e. \(\varphi_{m}(t)=\varphi_{n}(t)\) for all \(t\geq{s}\));
  \item \label{item:denseness} The sequence $((s_n,x_n):n\in \mathbb{N})$ is
    dense in $\mathbb{R}\times M$;
  \item \label{item:rel_comp} For any compact
    $L\subset \mathbb{R}\times M,$ the set
    $\{\varphi{}_n[s,\infty):\ n\in I^s,\, (s,\varphi_n(s))\in L\}$ is
    relatively compact in $X$.
  \end{enumerate}

\end{definition}

Let us denote by \(\mathcal{S}(M)\) the space of all skeletons on $M$.
The space $X^{\mathbb{N}}$, equipped with the product topology, is a
separable completely metrizable space (see
\cite[Th. 2.4.3]{zbMATH01166155}). Note that $\mathcal{S}(M)$ is a
Borel subset of $X^{\mathbb{N}}$. We equip \(\mathcal{S}(M)\) with the
subspace topology induced from $X^\mathbb{N}$ and its corresponding
Borel \(\sigma\)-field. Thus, the $\sigma$-field on $\mathcal{S}(M)$
is generated by the mappings $\varphi\mapsto i(\varphi_n),$
$\varphi\mapsto e(\varphi_n)(t), \ n\in \mathbb{N}, \ t\in\mathbb{R}$.

\smallskip{}Let us prove the following lemma that is satisfied by
every skeleton \(\varphi\). Throughout the paper $B(x,r)$ (resp.,
$\overline{B}(x,r)$) will denote an open (resp., closed) ball in $M$
with center $x$ and radius $r$.
\begin{lemma}
  \label{lem:past_dense}
  If \(\varphi\in\mathcal{S}(M)\), then for all \(s\in\mathbb{R}\) and
  $\delta>0$ the set
  \(\{\varphi_n(s): n\in I^s\setminus I^{s-\delta}\}\) is dense in
  $M$.
\end{lemma}
\begin{proof}
  Let us fix \(s\in\mathbb{R}, x\in{M}\) and \(\varepsilon>0\). By
  (Sk~\ref{item:rel_comp}) the set of trajectories
  \(\{\varphi_n: (s_n,x_n)\in (s-\delta,s)\times B(x,\varepsilon)\}\)
  is relatively compact in $X$. In particular, there exists
  $\alpha\in(0,\delta)$ such that for any $n$ with
  $(s_n,x_n)\in (s-\delta,s)\times B(x,\varepsilon)$, we have
  $\sup_{t\in
    [s_n,s_n+\alpha]}\rho(\varphi_n(t),x_n)<\frac{\varepsilon}{2}$. Using
  (Sk~\ref{item:denseness}), there exists \(n\in\mathbb{N}\) such that
  $(s_n,x_n)\in (s-\alpha,s)\times B(x,\frac{\varepsilon}{2})$. Then
  we get that 
  $\rho(\varphi_n(s),x)<\varepsilon$.
\end{proof}

\subsection{The measurable mapping \(\Theta\)}
\label{subsection:skeleton_extension}
From Lemma \ref{lem:selection} stated in Appendix~\ref{sec:selection},
it follows that there exists a measurable mapping
$\ell:X^{\mathbb{N}}\to X$ such that for any relatively compact
sequence $(f_n:n\in \mathbb{N})$ in X, $\ell((f_n:n\in \mathbb{N}))$
is a limit point of this sequence. We fix such a mapping from now on.

Let $(\varepsilon_k:k\in \mathbb{N})$ be a sequence of measurable
functions from \(M\) into \((0,\infty)\) such that for every $x\in M$,
the sequence \((\varepsilon_{k}(x): k \in \mathbb{N})\) is
non-increasing and converges to \(0\) as \(k\to\infty\).
For all $(s,x)\in \mathbb{R}\times M$ and all \(k\in \mathbb{N}\), set
\begin{equation}
  \label{eq:approximating_sequence}
  \left\{
    \begin{aligned}
      &n^{s,x}_k=\inf \{n\in I^s:
        \rho(\varphi_n(s),x)<\varepsilon_k(x)\},\\
      &\varphi^{s,x}_k=\varphi_{n^{s,x}_k}[s,\infty).
    \end{aligned}
  \right.
\end{equation}

Lemma~\ref{lem:past_dense} ensures that \(n^{s,x}_{k}<\infty\) for all
\((s,x)\in\mathbb{R}\times{M}\) and all \(k\in\mathbb{N}\). This
defines a sequence $(\varphi^{s,x}_k:k\in \mathbb{N})$ in
$C([s,\infty):M)$. Let us now define the mapping \(\Theta\) such that
for all $(s,x,\varphi)\in\mathbb{R}\times{M}\times\mathcal{S}(M)$,
\begin{displaymath}
  \Theta(s,x,\varphi)=\ell\big((\varphi^{s,x}_k:k\in \mathbb{N})\big).
\end{displaymath}

\begin{theorem}
  \label{thm:measurable_extension}
  The mapping $\Theta:\mathbb{R}\times{M}\times\mathcal{S}(M)\to{}X$
  is measurable and satisfies the following properties for all
  \((s,\varphi)\in\mathbb{R}\times\mathcal{S}(m)\):
  \begin{itemize}
  \item For all $x\in M$, $\Theta(s,x,\varphi)\in{}C([s,\infty):M)$
    and $\Theta(s,x,\varphi)(s)=x$.
  \item For all $n\in I^s$,
    $\Theta(s,\varphi_n(s),\varphi)=\varphi_n[s,\infty).$
    
  \end{itemize}
\end{theorem}

\begin{proof}
  Let us fix
  \((s,x,\varphi)\in\mathbb{R}\times{M}\times\mathcal{S}(M)\). By
  (Sk~\ref{item:rel_comp}) the sequence
  $(\varphi^{s,x}_k:k\in\mathbb{N})$ is relatively compact in
  $C([s,\infty):M)$ with $\lim_{k\to\infty}\varphi^{s,x}_k(s)=x$. This
  implies that for all
  $(s,x,\varphi)\in\mathbb{R}\times{M}\times\mathcal{S}(M)$,
  $\Theta(s,x,\varphi)\in{}C([s,\infty):M)$ and
  $\Theta(s,x,\varphi)(s)=x$.

  We now verify that $\Theta$ preserves the trajectories of the
  skeleton. Let us fix $(s,\varphi)\in\mathbb{R}\times\mathcal{S}(M)$
  and let \(n\in I^{s}\). Set \(x=\varphi_{n}(s)\). Then the sequence
  \((n^{s,x}_{k}:k\in\mathbb{N})\) is stationary and
  \(\Theta(s,x,\varphi)=\varphi_{n^{*}}[s,\infty)\) for some
  \(n^{*}\in{}I^{s}\). We have that
  $\varphi_{n^*}(s)=\varphi_n(s)=x$. Now, property
  (Sk~\ref{item:coalescence}) implies that
  \(\varphi_{n}[s,\infty)=\varphi_{n^{*}}[s,\infty)=\Theta(s,x,\varphi)\).

  It remains to show that $\Theta$ is measurable.  Let us fix
  \(k\in\mathbb{N}\). The mapping \((s,x,\varphi)\mapsto{}n^{s,x}_{k}\) is
  measurable since for all \(m\in\mathbb{N}\),
  \begin{equation*}
    \{(s,x,\varphi):n^{s,x}_{k}\geq{m}\} =
    \cap_{n<m} \{ (s,x,\varphi) :
    s_n\leq{s} \Rightarrow \rho\left(\varphi_{n}(s),x\right)\geq \varepsilon_k(x)
    \}.
  \end{equation*}
  Clearly, \(m\mapsto\varphi_{m}\) is measurable. As a consequence,
  \((s,x,\varphi)\mapsto\varphi^{s,x}_{k}\) is measurable, and the
  measurability of the mapping \(\ell\) implies that \(\Theta\) is
  measurable.
\end{proof}

\subsection{Bifurcation points}
\label{subsection:bifurcation_points}

Let $\Theta$ be the measurable mapping constructed in
Section~\ref{subsection:skeleton_extension} out of a given sequence
\((\varepsilon_{k}:k\in\mathbb{N})\), and let
\(\varphi\in\mathcal{S}(M)\) be a skeleton. For \(s\leq t\) and
\(x\in M\), set $\theta_{s,t}(x)=\Theta(s,x,\varphi)(t)$. Then
$\theta_{s,\cdot}(x)\in C([s,\infty):M)$ defines a family of
continuous mappings that preserves the skeleton \(\varphi\) (i.e.,
$\theta_{s,s}(x)=x$ and $\theta_{s,t}(\varphi_n(s))=\varphi_n(t)$ if
$s_{n}\leq s\leq t$). Note that the mapping
$(s,t,x,\varphi)\mapsto \theta_{s,t}(x)$ is measurable.

For \(s<t\), \(\varepsilon>0\) and \(x\in{M}\), denote by
\(\mathcal{K}^{s,t}_{\varepsilon,x}\) the closure of
\(\{\varphi_n[s,t]:\,n\in{}I^{s}, \varphi_n(s)\in B(x,\varepsilon)\}\)
in $C([s,t]:M)$. Assumption (Sk~\ref{item:rel_comp}) implies that
\(\mathcal{K}^{s,t}_{\varepsilon,x}\) is a compact subset of
\(C([s,t]:M)\) for all \(\varepsilon>0\). Lemma \ref{lem:past_dense}
implies that \(\mathcal{K}^{s,t}_{\varepsilon,x}\) is
non-empty. Therefore,
\(\mathcal{K}^{s,t}_{x}:=\cap_{\varepsilon>0}\mathcal{K}^{s,t}_{\varepsilon,x}\)
is a non-empty compact subset of \(C([s,t]:M)\). In fact,
$\mathcal{K}^{s,t}_x$ contains the restriction
$\theta_{s,\cdot}(x)[s,t].$ For \(s<t\), set
\(\nu^{s,t}_{x}:=\#\mathcal{K}^{s,t}_{x}\) (with $\nu^{s,t}_x=\infty,$
if $\mathcal{K}^{s,t}_x$ is infinite).

\begin{lemma}
  For all \((s,x)\), the mapping \(t\mapsto{}\nu^{s,t}_{x}\) is
  left-continuous and non-decreasing.
\end{lemma}
\begin{proof}
  Let us be given \(s<t<u\) and \(x\in{M}\). If
  \(f\in\mathcal{K}^{s,u}_{x}\), then the restriction \(f[s,t]\) of
  \(f\) to \([s,t]\) belongs to \(\mathcal{K}^{s,t}_{x}\). The mapping
  \(f\mapsto{}f[s,t]\) is a surjective mapping from
  \(\mathcal{K}^{s,u}_{x}\) onto \(\mathcal{K}^{s,t}_{x}\), which
  shows that \(\nu^{s,t}_{x}\leq{}\nu^{s,u}_{x}\) and so
  \(t\mapsto{}\nu^{s,t}_{x}\) is non-decreasing.

  It remains to prove that \(t\mapsto{}\nu^{s,t}_{x}\) is
  left-continuous. Let $t>s$ and let $m$ be an integer such that
  $1\leq m\leq \nu^{s,t}_x$. There are at least $m$ distinct functions
  $f_1,\ldots,f_{m}\in \mathcal{K}^{s,t}_x.$ By continuity, for some
  $r\in (s,t)$ the restrictions
  $f_1[s,r],\ldots,f_{m}[s,r]\in \mathcal{K}^{s,r}_x$ are all
  distinct. Hence, $\nu^{s,t-}_x\geq \nu^{s,r}_x\geq m.$ It follows
  that $\nu^{s,t-}_x=\nu^{s,t}_x.$
\end{proof}

For \((s,x)\in\mathbb{R}\times{M}\), set
$\tau^s_x=\inf\{t> s: \nu^{s,t}_x\geq 2\}.$

\begin{definition}
\label{def:bifurcation_points}
The set \(B(\varphi):=\{(s,x)\in\mathbb{R}\times{M}:\,\tau^s_x=s\}\)
is called the set of bifurcation points of the skeleton $\varphi$.
\end{definition}

For every $(s,x)\in \mathbb{R}\times M$ and every $t\in (s,\tau^s_x)$
the set $\mathcal{K}^{s,t}_x$ contains a single function
$\theta_{s,\cdot}(x)[s,t]$.

\begin{lemma}
  \label{lem:hitting_B}
  If $\tau^s_x<\infty,$ then
  $\left(\tau^s_x,\theta_{s,\tau^s_x}(x)\right)\in B(\varphi)$.
\end{lemma}

\begin{proof} For any $t>\tau^s_x$ we have $\nu^{s,t}_x\geq 2$ and
  there exist at least two distinct functions
  $g_1,g_2\in \mathcal{K}^{s,t}_x.$ Let $(n^1_k:k\geq 1)$ and
  $(n^2_k:k\geq 1)$ be two sequences in $I^s$ such that
  \(\varphi_{n^i_k}[s,t]\) converges uniformly towards \(g_i\),
  \(i\in\{1,2\}\). Since
  $g_1[s,\tau^s_x]=g_2[s,\tau^s_x]=\theta_{s,\cdot}(x)[s,\tau^s_x]$,
  it follows that $g_1[\tau^s_x,t]\ne g_2[\tau^s_x,t].$ Further,
  $\lim_{k\to\infty}\varphi_{n^i_k}(\tau^s_x)=\theta_{s,\tau^s_x}(x).$
  Since $\varphi_{n^i_k}[\tau^s_x,t]\to g_i[\tau^s_x,t]$ uniformly on
  $[\tau^s_x,t],$ it follows that
  $g_i[\tau^s_x,t]\in
  \mathcal{K}^{\tau^s_x,t}_{\theta_{s,\tau^s_x}(x)}.$ Hence,
  $\nu^{\tau^s_x,t}_{\theta_{s,\tau^s_x}(x)}\geq 2$ for all
  $t>\tau^s_x.$
\end{proof}

Consider a closed set $F\supset B(\varphi)$. We define for every
\((s,x)\in \mathbb{R}\times M\) a sequence
\(\left((\sigma^{s}_{x}(k),z^{s}_{x}(k)): k\geq 0\right)\) in
$[s,\infty]\times M.$ Let \(\sigma^{s}_{x}(0)=s\) and
\(z^{s}_{x}(0)=x\). For \(k\geq 0\) let
\begin{align*}
  & \sigma^{s}_{x}(k+1)=
    \begin{cases}
      \inf
      \left\{ t>\sigma^{s}_{x}(k):\:
      \left(t,\theta_{\sigma^{s}_{x}(k),t}(z^{s}_{x}(k))\right)\in{F}
      \right\} &\mbox{if}\quad  \sigma^s_x(k)<\infty
      \\
      \infty  &\mbox{if}\quad   \sigma^s_x(k)=\infty
    \end{cases}
  \\
  & z^{s}_{x}(k+1)=
    \begin{cases}
      \theta_{\sigma^{s}_{x}(k),\sigma^{s}_{x}(k+1)}(z^{s}_{x}(k))
      &\mbox{if}\quad \sigma^{s}_{x}(k+1)<\infty
      \\
      x  &\mbox{if}\quad \sigma^{s}_{x}(k+1)=\infty
    \end{cases}
\end{align*}
For all \((s,x)\in \mathbb{R}\times M\),
\((\sigma^{s}_{x}(k):k\geq 0)\) is a non-decreasing sequence in
$[s,\infty].$ If \(\sigma^{s}_{x}(k)=\infty\), then
$\sigma^s_x(j)=\infty$ and $z^s_x(j)=x$ for all $j\geq k$. Let us also
define
\begin{displaymath}
  k^{s}_{x}:=\inf\{k\geq 0: \:
  \sigma^{s}_{x}(k)=\sigma^{s}_{x}(k+1)\}.
\end{displaymath}
For all
\(j\geq k^{s}_{x}\), \(\sigma^{s}_{x}(j)=\sigma^{s}_{x}(k^{s}_{x})\)
and \(z^{s}_{x}(j)=z^{s}_{x}(k^{s}_{x})\). Set
\(\sigma^{s}_{x}(\infty)=\lim_{k\to\infty}\sigma^{s}_{x}(k)\), which
is equal to \(\sigma^{s}_{x}(k^{s}_{x})\) when \(k^{s}_{x}<\infty\).

\begin{lemma}
  \label{lem:sigma_z_measurability}
  For all $k\geq 0$, the mapping $(s,x,\varphi)\mapsto (\sigma^s_x(k),
  z^s_x(k))$ is measurable.
\end{lemma}

\begin{proof}
  For $k=0$ the statement clearly holds. We proceed by induction on
  $k.$ The inequality $\sigma^s_x(k+1)>t$ holds if and only if either
  $\sigma^s_x(k)>t$ or $\sigma^s_x(k)\leq t$ and
  $(r,\theta_{\sigma^s_x(k)},r)(z^s_x(k))\not\in F$ for all
  $r\in (\sigma^s_x(k),u)$ with some $u>t.$ Measurability of
  $z^s_x(k+1)$ follows from measurability of $\sigma^s_x(k),$
  $\sigma^s_x(k+1),$ $z^s_x(k)$ and measurability of $\theta_{s,t}(x)$
  as a function of $(s,t,x,\varphi)$.
\end{proof}

Out of \(\theta\), we now define for all
\((s,x)\in \mathbb{R}\times M\) a function
\(\psi_{s,\cdot}(x):[t,\infty)\to\infty\): let \(t\geq s\).
\begin{itemize}
\item If \(t<\sigma^{s}_{x}(\infty)\), then there is \(k<k^{s}_{x}\) such that
  \(t\in[\sigma^{s}_{x}(k),\sigma^{s}_{x}(k+1))\) and we set
  \begin{displaymath}
    \psi_{s,t}(x)=\theta_{\sigma^{s}_{x}(k),t}(z^{s}_{x}(k)).
  \end{displaymath}
\item If \(t\geq\sigma^{s}_{x}(\infty)\) and \(k^{s}_{x}<\infty\),
  then \(\sigma^{s}_{x}(\infty)=\sigma^{s}_{x}(k^{s}_{x})\)
  and we set
  \begin{displaymath}
    \psi_{s,t}(x)=\theta_{\sigma^{s}_{x}(k^{s}_{x}),t}(z^{s}_{x}(k^{s}_{x})).
  \end{displaymath}
\item If \(t\geq\sigma^{s}_{x}(\infty)\) and \(k^{s}_{x}=\infty\),
  then we set \(\psi_{s,t}(x)=x\).
\end{itemize}
The function \(\psi_{s,\cdot}(x)\) is continuous if
\(\sigma^{s}_{x}(\infty)=\infty\), or if \(k^{s}_{x}<\infty\). 

\medskip{}

The next theorem gives a sufficient condition under which $\psi$ is a
strong flow.

\begin{theorem}
  \label{thm:sfp_deterministic}
  Suppose that for any
  \((s,x)\in \mathbb{R}\times M\),
  \begin{itemize}
  \item  \(k^{s}_{x}<\infty\);
  \item if 
  $k^s_x=0$, then for any \(t>s\) there is
    \(n\in I^{t}\) such that
    \(\theta_{s,\cdot}(x)[t,\infty)=\varphi_{n}[t,\infty)\).
  \end{itemize}
  Then \(\psi\) is a family of continuous mappings that satisfies the
  strong flow property and preserves the skeleton \(\varphi\), i.e.
  \begin{enumerate}[(i)]
  \item for all \((s,x)\in \mathbb{R}\times{M}\),
    \(\psi_{s,\cdot}(x)\in C([s,\infty):M)\) and
    \(\psi_{s,s}(x)=x\);
  \item for all \(s\leq t\) and \(n\in{I^{s}}\),
    $\psi_{s,t}\left(\varphi_n(s)\right)=\varphi_n(t)$.
  \item\label{item:21} for all \(s\leq t\leq u\),
    $\psi_{t,u}\circ\psi_{s,t}=\psi_{s,u}$.
  \end{enumerate}
\end{theorem}

\begin{proof}[Proof of Theorem~\ref{thm:sfp_deterministic}]
  The fact that \(\psi\) preserves the skeleton \(\varphi\) follows
  from the fact that \(\theta\) preserves the skeleton \(\varphi\). To
  prove the theorem, it suffices to verify that $\psi$ satisfies the
  strong flow property \eqref{item:21}, i.e. that for all \((s,x)\) it
  holds that
  \begin{equation}
    \label{eq:flow-ppty}
    \psi_{t,u}\circ\psi_{s,t}(x)=\psi_{s,u}(x), \quad \forall u\geq
    t\geq s.
  \end{equation}
  For \(d\geq 0\), set $S_d=\{(s,x)\in \mathbb{R}\times M:
  k^s_x=d\}$. We claim that for all $d\geq 0$ it holds that
  \eqref{eq:flow-ppty} is satisfied for all $(s,x)\in S_d$. We prove
  this claim by induction on $d$.
  
  Assume that $(s,x)\in S_0$. Then
  $\psi_{s,\cdot}(x)=\theta_{s,\cdot}(x)$ and the assumption of the
  theorem implies that for all $t>s$, $\psi_{s,t}(x)=\varphi_{n}(t)$
  for some $n\in I^t$. This implies that for all $u>t$
  $$
  \psi_{s,u}(x)=\varphi_n(u)=\psi_{t,u}(\varphi_n(t))=\psi_{t,u}(y).
  $$

  Let \(d\geq 1\) and suppose that \eqref{eq:flow-ppty} is satisfied
  for \((s,x)\in\cup_{k=0}^{d-1}S_{k}\). Let $(s,x)\in S_d$. Let
  \(t>s\) and set \(y=\psi_{s,t}(x)\). If $t<\sigma^s_x(1)$
  then $(t,y)\not\in F$. Hence,
  $t=\sigma^t_y(0)<\sigma^t_y(1)\leq \tau^t_y$ and
  $$
  \psi_{t,\cdot}(y)[t,\sigma^t_y(1))
  =\theta_{t,\cdot}(y)[t,\sigma^t_y(1))
  =\theta_{s,\cdot}(x)[t,\sigma^t_y(1)).
  $$ 
  This implies that $\sigma^s_x(1)=\sigma^t_y(1)$ and
  $\psi_{s,\cdot}(x)[t,\sigma^s_x(1))=\psi_{t,\cdot}(y)[t,\sigma^s_x(1))$.
  By construction, we obtain that
  $(\sigma^s_x(j),z^s_x(j))=(\sigma^t_y(j),z^t_y(j))$ for all
  $j\geq 1$. The definition of $\psi$ then implies that
  $\psi_{s,u}(x)=\psi_{t,u}(y)$ for all $u\geq \sigma^s_x(1)$.

  If $t\geq \sigma^s_x(1)$. Denote $r=\sigma^s_x(1)$ and
  $w=z^s_x(1)$. Then $(r,w)\in S_{d-1}$ and
  $\psi_{s,\cdot}(x)[r,\infty)=\psi_{r,\cdot}(w)$. Since~\eqref{eq:flow-ppty}
  is satisfied by \((r,w)\), we obtain that
  $$
  \psi_{s,u}(x)=\psi_{r,u}(w)=\psi_{t,u}(\psi_{r,t}(w))=\psi_{t,y}(y).
  $$
  We have thus proved that~\eqref{eq:flow-ppty} is satisfied for all
  \((s,x)\in\cup_{d=0}^{\infty}S_{d}=\mathbb{R}\times M\),
  i.e.~\eqref{item:21} is satisfied.   
\end{proof}

\section{Strong measurable continuous modifications of stochastic
  flows}
\label{sec:strong-flow-meas}

Let $\psi^0$ be a stochastic flows of measurable mappings in $M$
associated to a consistent sequence
$(\mathsf{P}^{(n)}_{\bullet}:n\in \mathbb{N})$ of coalescing Feller transition
functions on $M$. We assume that the sequence
$(\mathsf{P}^{(n)}_{\bullet}:n\in \mathbb{N})$ satisfies properties
(TF~\ref{item:consistency}), (TF~\ref{item:coalescing_property}),
(TF~\ref{item:continuity_of_trajectories_TF3}) and (TF~\ref{item:14})
(the condition (TF~\ref{item:14}) is introduced in Section
\ref{sec:2.3.3}). We define a random skeleton $\Phi$ and using the
mapping \(\Theta\) from Section~\ref{subsection:skeleton_extension},
we construct a measurable continuous modification $\theta$ of
\(\psi^{0}\). In Section \ref{sec:2.3.3} we state sufficient
conditions under which the construction of
Section~\ref{subsection:bifurcation_points} is a.s. applicable to
$\theta$ and produces a strong measurable continuous modification of
$\psi^0$.

\subsection{A measurable modification of a stochastic flow}
\label{sec:measurable_modification_out_of_skeleton}

Let $(\mathsf{P}^{(n)}_{\bullet}:n\in \mathbb{N})$ be a consistent sequence of
coalescing Feller transition functions on $M$ that satisfies
(TF~\ref{item:consistency}), (TF~\ref{item:coalescing_property}) and
(TF~\ref{item:continuity_of_trajectories_TF3}), and let $\psi^0$ be a
corresponding measurable stochastic flow of mappings in \(M\). For
\(n\in \mathbb{N}\) and \(x\in{M^{n}}\), we will denote by
\(\mathbb{P}^{(n)}_{x}\) the distribution on $C([0,\infty):M^n)$ under
which the canonical process \(X^{(n)}=(X^1,\ldots,X^n)\) is a Markov
process with transition function $\mathsf{P}^{(n)}$ and starting point
$x$. When \(n=1\) and \(n=2\), we will denote \(X^{(1)}\) and
\(X^{(2)}\) by \(X\) and \((X,Y)\), respectively.

\subsubsection{A first lemma}
Before formulating condition (TF~\ref{item:14}), we deduce some
consequences of (TF~\ref{item:continuity_of_trajectories_TF3}) that
will be used later.

\begin{lemma}
  \label{lem:stronger_TF3} For any compact $K\subset M$ and any $r>0$,
  as \(t\to\infty\)
  \[
    \sup_{x\in{K}}\mathbb{P}^{(1)}_{x}\left[\sup_{s\in[0,t]}\rho(X_{s},x)>r\right]=o(t).
  \]
\end{lemma}

\begin{proof}

  The proof follows the one given in \cite{zbMATH06929104} for the
  case $M=\mathbb{R}$. Let us fix a compact \(K\subset{M}\) and
  \(r>0\). By (TF~\ref{item:continuity_of_trajectories_TF3}) for any
  $\alpha>0$ there exists $\delta>0$ such that for all
  \(t\in[0,\delta]\) and \(x\in{M}\) with \(\rho(x,K)\leq{r}\), we
  have
  \begin{equation*}
    \mathbb{P}^{(1)}_{x}\left[\rho(X_{t},x)>\frac{r}{2}\right]\leq \alpha t.
  \end{equation*}
  Then, for any $x\in K$ we introduce the stopping time
  $\tau=\inf\{t>0:\rho(X_{t},x)\geq{r}\}$ and estimate
  \begin{align*}
    \mathbb{P}^{(1)}_{x}\left[\sup_{s\in[0,t]}\rho(X_{s},x)\geq{r}\right]
    &= \mathbb{P}^{(1)}_{x}\left[\{\tau\leq{t}\} \cap 
      \left\{\rho(X_{t},x)\geq\frac{r}{2}\right\}\right]\\
    &+ \mathbb{P}^{(1)}_{x}\left[\{\tau\leq{t}\} \cap 
      \left\{\rho(X_{t},x)<\frac{r}{2}\right\}\right]\\
    &\leq
      \mathbb{P}^{(1)}_{x}\left[\rho(X_{t},x)\geq\frac{r}{2}\right]
      + \mathbb{E}^{(1)}_{x}
      \left[1_{\{\tau\leq{t}\}}\mathbb{P}^{(1)}_{x}
      \left[
      \left.\rho(X_{t},X_{\tau})\geq\frac{r}{2}\right|\mathcal{F}_{\tau}\right]
      \right]\\
    &\leq 2\alpha{t},
  \end{align*}
  where we have used the strong Markov property at time \(\tau\) and
  the fact that \(\rho(X_{\tau},K)\leq{r}\) when \(\tau<\infty\).
\end{proof}

\subsubsection{The random sequence \(\varPhi\)}
\label{sec:rand-sequ-varphi}

Let \(\mathcal{D}=\{(s_{n},x_{n}):n\in \mathbb{N}\}\) be a dense
countable set in $\mathbb{R}\times M$. Let \(\varPhi_{n}\) be a
continuous modification of $\psi^0_{s_n,\cdot}(x_n)$, such that
$\varPhi_n(s_n)=x_n$. The sequence
\(\varPhi:=(\varPhi_{n}:n\in\mathbb{N})\) is a random element in
\(X^{\mathbb{N}}\). Its distribution will be denoted
\(\mathbb{P}^{\infty}_{\mathcal{D}}\).

We remark that, with probability 1, the sequence \(\varPhi\)
satisfies condition (Sk~\ref{item:coalescence}) (by
(TF~\ref{item:consistency})). Condition (Sk~\ref{item:denseness}) is
satisfied by construction. Using
(TF~\ref{item:continuity_of_trajectories_TF3}), we can prove the
following lemma.
   
\begin{lemma}
  \label{lem:denseness}
  Almost surely, for all \(t\in\mathbb{R}\) the set
  \(\{\varPhi_{n}(t):s_{n}<t\}\) is dense in \(M\).
\end{lemma}

\begin{proof}
  In this proof, we let \(\{y_{i}:i\in\mathbb{N}\}\) be a dense
  sequence in \(M\).

  For \(a<b\), let
  \(\Omega_{[a,b]}=\{\omega\in \Omega: \ \exists t\in [a,b] \
  \{\varPhi_n(\omega,t):s_n<t\} \mbox{ is not dense in } M\}.\) We
  will prove that the set $\Omega_{[a,b]}$ is negligible,
  i.e. $\mathbb{P}^*\left[\Omega_{[a,b]}\right]=0,$ where
  $\mathbb{P}^*$ is an outer measure that corresponds to $\mathbb{P}.$
  On \(\Omega_{[a,b]}\), there are \(t\in[a,b]\),
  \(y\in\{y_{i}:i\in\mathbb{N}\}\) and \(r\in\mathbb{Q}_{+}^{*}\) such
  that \(\{\varPhi_{n}(t):s_{n}<t\}\cap{B(y,2r)}=\emptyset\). Then for
  all \(\delta>0\) and all \(n\in\mathbb{N}\) such that
  \(t-2\delta<s_{n}<t\) and such that \(\rho(x_{n},y)<r\) we have that
  \(\rho(\varPhi_{n}(t),x_{n})>r\) and as a consequence that
  \(\sup_{s\in[0,2\delta]}\rho(\varPhi_{n}(s_n+s),x_{n})>r\). This
  implies that on \(\Omega_{[a,b]}\) there are
  \(y\in\{y_{i}:i\in\mathbb{N}\}\) and \(r\in\mathbb{Q}_{+}^{*}\) such
  that for all \(\delta>0\), there is
  \(j\in\mathbb{Z}\cap[\delta^{-1}{a}-1,\delta^{-1}{b})\) such that
  for all \(n\in\mathbb{N}\),
  \begin{equation}
    \label{eq:9}
    s_{n}\in ((j-1)\delta,j\delta]\quad\mathrm{and}\quad\rho(x_{n},y)<r
    \quad\Longrightarrow\quad
    \sup_{s\in[0,2\delta]}\rho(\varPhi_{n}(s_{n}+s),x_{n})>r.
  \end{equation}
  For \(\delta>0\), \(y\in\{y_{i}:i\in\mathbb{N}\}\),
  \(r\in\mathbb{Q}_{+}^{*}\) and
  \(j\in\mathbb{Z}\cap[\delta^{-1}{a}-1,\delta^{-1}{b}]\), denote by
  \(\Omega^{\delta,j}_{y,r}\) the event ``\eqref{eq:9} is satisfied
  for all \(n\in\mathbb{N}\)''. Then
  \begin{equation}
    \Omega_{[a,b]}\subset \bigcup_{y,r}\bigcap_{\delta>0}\bigcup_j \Omega^{\delta,j}_{y,r}.
  \end{equation}
  Since \(\{(s_{n},x_{n})_{n\in\mathbb{N}}\}\) is dense in
  \(\mathbb{R}\times{M}\), there is an \(n\) such that
  \(s_{n}\in ((j-1)\delta,j\delta]\) {and} \(\rho(x_{n},y)<r\), and so
  \begin{align*}
    \mathbb{P}[\Omega^{\delta,j}_{y,r}]
    &\leq
      \mathbb{P}\left[\sup_{s\in[0,2\delta]}\rho(\varPhi_{n}(s_{n}+s),x_{n})>r\right]
      =\mathbb{P}^{(1)}_{x_{n}}\left[\sup_{s\in[0,2\delta]}\rho(X_{s},x_{n})>r\right]
    \\
    &\leq
      \sup_{x\in{B(y,r)}}\mathbb{P}^{(1)}_{x}\left[\sup_{s\in[0,2\delta]}\rho(X_{s},x)>r\right].
  \end{align*}
  Therefore
   
  $$
  \mathbb{P}\left[\bigcup_j \Omega^{\delta,j}_{y,r}\right] \leq
  \left(\delta^{-1}b-\delta^{-1}a+2\right)\sup_{x\in
    B(y,r)}\mathbb{P}^{(1)}_x\left[\sup_{s\in [0,2\delta]}
    \rho(X_s,x_n)>r\right],
  $$
  and
  \begin{equation}
    \mathbb{P}^*\left[\Omega_{[a,b]}\right]
    \leq \sum_{y,r}
    \lim_{\delta\downarrow{0}} \left(b-a+2\delta\right)\delta^{-1}\sup_{x\in B(y,r)}\mathbb{P}^{(1)}_x\left[\sup_{s\in [0,2\delta]} \rho(X_s,x_n)>r\right].
  \end{equation}
  Applying Lemma \ref{lem:stronger_TF3}, we can conclude that
  \(\mathbb{P}^*\left[\Omega_{[a,b]}\right]=0\). This proves the
  lemma.
\end{proof}

\subsubsection{Assumption (TF~\ref{item:14})}\label{sec:2.3.3}
The next assumption is formulated in terms of
$\mathbb{P}^\infty_{\mathcal{D}}$\,:
\begin{enumerate}[(TF~1)]
  \setcounter{enumi}{3}
\item\label{item:14}
  $\mathbb{P}^\infty_{\mathcal{D}}(\mathcal{S}(M))=1.$
\end{enumerate}
Another formulation of (TF~\ref{item:14}) is to say that if a random
variable \(\varPhi\) in $X^\mathbb{N}$ is distributed as
\(\mathbb{P}^{\infty}_{\mathcal{D}}\), then \(\varPhi\) is a skeleton
a.s.

Condition (TF~\ref{item:14}) is equivalent to the following condition:
\begin{enumerate}[(TF~1')]
  \setcounter{enumi}{3}
\item\label{tf_item:2} For any compact subset
  \(L\subset\mathbb{R}\times{M}\), the set
  \(\{\varPhi_{n}[s,\infty):\ n\in I^s, (s,\varPhi_n(s))\in L\}\) is
  a.s.  relatively compact in \(X\).
\end{enumerate}
\begin{remark}
  \label{rem:independence of assumptions} The authors don't know
  whether condition (TF~\ref{item:14}) follows from conditions
  (TF~\ref{item:consistency}), (TF~\ref{item:coalescing_property}) and
  (TF~\ref{item:continuity_of_trajectories_TF3}). In Theorem
  \ref{thm:inst_coalescence_metric_graph} we prove this fact for
  instantaneously coalescing stochastic flows on metric graphs.
\end{remark}

\subsubsection{A sequence of measurable mappings
  \((\varepsilon_{k}: k\in\mathbb{N})\)}
\label{sec:sequ-meas-mapp} We recall that \((X,Y)\) denotes the
canonical process on \(C([0,\infty):M^{2})\), so that under
\(\mathbb{P}^{(2)}_{(x,y)}\) the process \((X,Y)\) is a Markov process
with transition function $\mathsf{P}^{(2)}_{\bullet}$ and starting point
$(x,y).$ By $d_0$ we denote the restriction of the metric $d$ from $X$
to $C([0,\infty):M),$ and by $d_s$ we denote its shift:
$ d_s(f,g)=d_0(f(s+\cdot),g(s+\cdot)),$ $(f,g)\in C([s,\infty):M)^2.$

\begin{lemma}
  \label{lem:two_point_continuity}
  For any $\varepsilon>0$, the function
  $(x,y)\mapsto \mathbb{P}^{(2)}_{(x,y)}[d_0(X,Y)\geq{}\varepsilon]$
  is upper-semicontinuous on $M^2$. In particular,
  \begin{equation}
    \lim_{y\to x}\mathbb{P}^{(2)}_{(x,y)}[d_0(X,Y)\geq{}\varepsilon]=0.
  \end{equation}
\end{lemma}

\begin{proof} Let $(x_n,y_n)\to (x,y)$ in $M^2$. By \cite[Chapter 4,
  Theorem 2.5]{zbMATH02237386} we have the weak convergence
  $\mathbb{P}^{(2)}_{(x_n,y_n)}\to \mathbb{P}^{(2)}_{(x,y)}$ in the
  space $C([0,\infty):M^2)$. Hence,
  $$
  \limsup_{n\to\infty}\mathbb{P}^{(2)}_{(x_n,y_n)}[d_0(X,Y)\geq{}\varepsilon]\;\leq\;
  \mathbb{P}^{(2)}_{(x,y)}[d_0(X,Y)\geq{}\varepsilon].
  $$
  The second statement follows from the relation
  $\mathbb{P}^{(2)}_{(x,x)}[d_0(X,Y)\geq{}\varepsilon]=0$.
\end{proof}

We define for every \(k\in \mathbb{N}\) a mapping \(\varepsilon_{k}\) such
that for all \(x\in{M}\)
\begin{equation}
  \label{eq:choice_of_epsilon}
  \varepsilon_k(x)=
  2^{-k}\wedge  \inf
  \left\{r>0:
    \sup_{y\in \overline{B}(x,r)}\mathbb{P}^{(2)}_{(x,y)}[d_0(X,Y)\geq{}2^{-k}]>2^{-k}
  \right\}.
\end{equation}

\begin{lemma}\label{lem:epsilon_measurability}
  The sequence \((\varepsilon_{k}: k\in \mathbb{N})\) is a non-increasing
  sequence of positive measurable functions, that converges uniformly
  towards \(0\) as \(k\to\infty\).
\end{lemma}
\begin{proof}
  The facts that the sequence \((\varepsilon_{k}:k\in \mathbb{N})\) is
  non-increasing and converges to $0$ for every $x\in M$ easily follow
  from its definition.  Lemma~\ref{lem:two_point_continuity} implies
  that \(\varepsilon_{k}(x)>0\) for all \(k\in \mathbb{N}\) and \(x\in{M}\).
  Let us now prove that for every \(k\in \mathbb{N}\), the mapping
  \(\varepsilon_{k}\) is measurable.  For \(r>0\) and \(k\in \mathbb{N}\),
  let \(g_{k}:M^{2}\to[0,1]\) and \(g_{r,k}:M\to[0,1]\) be 
  functions defined by
  \(g_{k}(x,y)=\mathbb{P}^{(2)}_{(x,y)}[d_0(X,Y)\geq{}2^{-k}]\) and
  \(g_{r,k}(x)=\sup_{y\in{}\overline{B}(x,r)}g_{k}(x,y)\).  Then
  \(g_{k}\) is upper-semicontinuous
  (Lemma~\ref{lem:two_point_continuity}) and this implies that
  \(g_{r,k}\) is also upper-semicontinuous.

  For any $a\in (0,2^{-k}]$, the set
  \(\varepsilon_{k}^{-1}((-\infty,a))\) has the following
  decomposition:
  $$
  \{x\in{M}:\,\varepsilon_k(x)<a\}=
  \bigcup_{(r,s)\in\mathbb{Q}^{2}:\;r\in (0,a),\,s>2^{-k}}
  \{x\in{M}:\enspace g_{r,k}(x)\geq{}s\}.
  $$
  Since $g_k$ is upper-semicontinuous, the set
  $\{x\in{M}:\enspace g_{r,k}(x)\geq{}s\}$ is closed, and this proves
  that \(\varepsilon_{k}\) is measurable.
\end{proof}

\subsubsection{The measurable continuous modification \(\theta\)}
\label{sec:meas-modif}

Assume that conditions (TF~\ref{item:consistency}),
(TF~\ref{item:coalescing_property}),
(TF~\ref{item:continuity_of_trajectories_TF3}) and (TF~\ref{item:14})
are satisfied. In particular, the sequence $\varPhi$ is a.s. a
skeleton. We restrict to an event of full measure on which $\varPhi$
is a skeleton. Using the sequence of functions \((\varepsilon_{k}:k\in \mathbb{N})\)
given by \eqref{eq:choice_of_epsilon}, we let 
\(\Theta:\mathbb{R}\times M\times \mathcal{S}(M)\to X\) be the mapping
defined in Section
\ref{subsection:skeleton_extension}. Let us now define
\(\theta=(\theta_{s,\cdot}(x):(s,x)\in\mathbb{R}\times{M})\) the
family of random continous mappings on \(M\) such that for \(s\leq{t}\) and
\((x,\omega)\in{M}\times\Omega\),
\begin{equation}\label{eq:2}
  \theta_{s,t}(\omega,x)=\Theta(s,x,\varPhi(\omega))(t).
\end{equation}

\begin{theorem}
  \label{thm:measurable_modification}
  The family of random mappings \(\theta\) is a measurable continuous 
  modification of \(\psi^{0}\) that preserves the skeleton $\varPhi$.
\end{theorem}
\begin{proof} The facts that
  \((s,t,x,\omega)\mapsto\theta_{s,t}(\omega,x)\) is measurable and
  that $\theta_{s,s}(\omega,x)=x$ follow immediately from Theorem
  \ref{thm:measurable_extension}. Let us now prove that \(\theta\) is
  a modification of \(\psi^{0}\).

  Let us fix \((s,x)\in\mathbb{R}\times{M}\). Taking a continuous
  modification of $\psi^0_{s,\cdot}(x)$ we may assume that
  $\psi^0_{s,\cdot}(x)\in C([s,\infty):M)$.  Recall that
  $\Theta(s,x,\varphi)=\ell((\varphi^{s,x}_k:k\in \mathbb{N}))$, with
  $\varphi^{s,x}_k$ and $n^{s,x}_k$ defined
  by~\eqref{eq:approximating_sequence}. We denote by $\varPhi^{s,x}_k$
  the random variable in $X$ defined
  by~\eqref{eq:approximating_sequence} with \(\varphi\) replaced by
  \(\varPhi\).  It holds that for each \(k\in \mathbb{N}\),
  \(\rho(\varPhi^{s,x}_{k}(s),x)<\varepsilon_{k}(x)\), which implies
  that (see the definition of \(\varepsilon_{k}(x)\) given
  in~\eqref{eq:choice_of_epsilon})
  \begin{equation*}
    \mathbb{P}[d_s(\varPhi^{s,x}_{k},\psi^{0}_{s,\cdot}(x))\geq{}2^{-k}]\leq{}2^{-k},
  \end{equation*}
  since \((\varPhi^{s,x}_{k},\psi^{0}_{s,\cdot}(x))\) is a Markov
  process in \(M^{2}\) with transition function given by
  \(\mathsf{P}^{(2)}\). The Borel-Cantelli lemma shows that
  a.s. \(\varPhi^{s,x}_{k}\to\psi^{0}_{s,\cdot}(x)\) in \(X\). This
  implies that a.s. \(\Theta(s,x,\varPhi)=\psi^{0}_{s,\cdot}\) and so
  that for all \(t\geq{s}\),
  a.s. \(\theta_{s,t}(x)=\psi^{0}_{s,t}(x)\), which proves that
  \(\theta\) is a modification of \(\psi^{0}\). In particular,
  $\theta$ is a stochastic flow of measurable mappings in $M$ (as it
  has been remarked in the Introduction). The fact that \(\theta\)
  preserves the skeleton \(\varPhi\) follows from
  Theorem~\ref{thm:measurable_extension}.
\end{proof}

The measurable continuous modification $\theta$ possesses one useful
feature: it a.s. satisfies the flow property at stopping times.

Let
\(\mathcal{A}^{\theta}=(\mathcal{A}^{\theta}_{t}:t\in\mathbb{R})\)
be the natural filtration associated to \(\theta\), where for
every \(t\in\mathbb{R}\cup\{\infty\}\) the $\sigma$-field
\(\mathcal{A}^{\theta}_{t}\) is defined by
$\mathcal{A}^{\theta}_t=\sigma(\{\theta_{r,s}:-\infty<r\leq
s\leq t\})$.  An $\mathbb{R}\cup \{\infty\}$-valued random variable
$\sigma$ will be called an $\mathcal{A}^{\theta}$-stopping time,
if for all $t\in \mathbb{R},$
$\{\sigma< t\}\in \mathcal{A}^{\theta}_t$ (equivalently, for all
$t\in \mathbb{R}$,
$\{\sigma\leq t\}\in \mathcal{A}^{\theta}_{t+}$). With an
$\mathcal{A}^{\theta}$-stopping time $\sigma$ we associate a
$\sigma$-field
\[
  \mathcal{A}^{\theta}_\sigma=\{A\in
  \mathcal{A}^{\theta}_\infty:\: \forall{}t\in \mathbb{R},\, A\cap
  \{\sigma<t\}\in{}\mathcal{A}^{\theta}_t\}.
\]

\begin{lemma}\label{lem:stopping_time_flow}
  Let $(s,x)\in\mathbb{R}\times{M}$ and let $\sigma$ be an
  $\mathcal{A}^{\theta}$-stopping time. Then
  $$
  \theta_{\sigma,\cdot}(\theta_{s,\sigma}(x))=\theta_{s,\cdot}(x)[\sigma,\infty)
  $$
  a.s. on the set $\{s\leq \sigma<\infty\}$.
\end{lemma}

\begin{proof}
  Denote $\xi=\theta_{s,\sigma}(x)$. For \(k\geq{1}\), let \(g_{k}\)
  be the function defined in Lemma~\ref{lem:epsilon_measurability},
  i.e.  $g_k(x,y)=\mathbb{P}^{(2)}_{(x,y)}[d_0(X,Y)\geq
  2^{-k}]$. Then, for all $y\in B(x,\varepsilon_k(x))$, we have
  $g_k(x,y)\leq 2^{-k}$. Finite point motions of the flow $\theta$ are
  Feller processes, in the sense that for any $n\in \mathbb{N},$
  $(s_1,x_1),\ldots,(s_n,x_n)\in \mathbb{R}\times M$ and
  $s\geq s_1\vee\ldots\vee s_n$ the process
  $((\theta_{s_1,t}(x_1),\ldots,\theta_{s_n,t}(x_n)):t\geq s)$ is an
  $(\mathcal{A}^{\theta}_t:t\geq s)$-Markov process in $M^n$ with
  transition function $\mathsf{P}^{(n)}_{\bullet}$. Below we will use
  the equality $\varPhi_n=\theta_{s_n,\cdot}(x_n).$

  For any $m\in \mathbb{N}$ and \(k\geq{1}\), let us introduce the
  event
  $E_{m,k} =
  \{\sigma\in [s,\infty), n^{\sigma,\xi}_k=m\}
  \in \mathcal{A}^{\theta}_\sigma$. By the strong Markov property for
  the two-point motions of $\theta$ at time \(\sigma\) we find
  that
  \begin{align*}
    &\mathbb{P}\left[
      s\leq \sigma<\infty,
      d_\sigma(\varPhi^{\sigma,\xi}_k,\theta_{s,\cdot}(x)[\sigma,\infty))
      \geq 2^{-k} \right]
    \\
    &=\sum_{m} \mathbb{P}\left[
      E_{m,k} \bigcap
      \left\{
      d_\sigma(\varPhi_m[\sigma,\infty),\theta_{s,\cdot}(x)[\sigma,\infty))
      \geq 2^{-k}
      \right\}
      \right]
    \\
    &=\sum_{m} \mathbb{E}\left[
      1_{E_{m,k}}\mathbb{P}\left[
      d_\sigma(\theta_{s_m,\cdot}(x_m)[\sigma,\infty),\theta_{s,\cdot}(x)[\sigma,\infty))
      \geq 2^{-k}\bigg| \mathcal{A}^{\theta}_\sigma
      \right]
      \right]
    \\
    &=\sum_{m} \mathbb{E}\left[
      1_{E_{m,k}}\mathbb{P}\left[
      d_0(\theta_{s_m,\sigma+\cdot}(x_m),\theta_{s,\sigma+\cdot}(x))
      \geq 2^{-k}\bigg| \mathcal{A}^{\theta}_\sigma
      \right]
      \right]
    \\
    &=\sum_{m}
      \mathbb{E}\left[
      1_{E_{m,k}}g_{k}(\theta_{s_m,\sigma}(x_m),\xi)
      \right]
      \leq 2^{-k}.     
  \end{align*}
  The latter estimate follows from the fact that on the event
  $E_{m,k}$ we have
  $\rho(\theta_{s_m,\sigma}(x_m),\xi)<\varepsilon_k(\xi)$. By the
  Borel-Cantelli lemma, a.s. on the set $\{s\leq \sigma<\infty\}$ the
  sequence $(\varPhi^{\sigma,\xi}_k:k\geq 1)$ converges to
  $\theta_{s,\cdot}(x)[\sigma,\infty)$. Hence, on this set
  a.s.
  $\theta_{\sigma,\cdot}(\xi)=\theta_{s,\cdot}(x)[\sigma,\infty)$
  and the lemma is proved.
\end{proof}

\subsection{Existence of a strong measurable continuous modification}
\label{sec:sf-meas-modif}

Assume that conditions (TF~\ref{item:consistency}),
(TF~\ref{item:coalescing_property}),
(TF~\ref{item:continuity_of_trajectories_TF3}) and (TF~\ref{item:14})
are satisfied. In particular, the sequence $\varPhi$ is a.s. a
skeleton. 

\begin{definition}
  \label{def:measurable_shell} A closed set
  $F\subset\mathbb{R}\times M$ is called a closed shell of the set of
  bifurcation points of $\varPhi,$ if a.s. $B(\varPhi)\subset F.$
\end{definition}

Let $F$ be a closed shell of the set of bifurcation points of
$\varPhi$. We define stopping times $\sigma^s_x(k),$ random variables
$z^s_x(k),$ $k^s_x,$ and random functions $\psi_{s,\cdot}(x)$ as in
Section~\ref{subsection:bifurcation_points}.

\begin{theorem}
\label{thm:sfp_stochastic} Assume that a.s. 

\begin{itemize}
\item for all $(s,x)\in \mathbb{R}\times M$,  $k^s_x<\infty$;

\item for all $(s,x)\in \mathbb{R}\times M$, if $k^s_x=0$ then for
  every $t>s$ there is $n\in I^t$ such that
  $\theta_{s,\cdot}(x)[t,\infty)=\varPhi_n[t,\infty)$.
\end{itemize}
Then $\psi$ is a strong measurable continuous modification of  $\psi^0$.
\end{theorem}

\begin{proof}
  Using Theorem~\ref{thm:sfp_deterministic} one obtains that
  a.s. \(\psi\) is a random family of continuous mappings that
  satisfies the strong flow property and preserves \(\varPhi\). It
  remains only to check that $\psi$ is a modification of $\theta$,
  which follows from Lemma \ref{lem:stopping_time_flow} and the
  definition of \(\psi\) out of \(\theta\).
\end{proof}

\section{The instantaneous coalescence property}
\label{sec:icp}

In this section we study skeletons $\varphi$ in $M$ that possess the
instantaneous coalescing property (ICP). By $\psi^0$ we denote a
stochastic flow of measurable mappings in $M$ associated to a
consistent sequence $(\mathsf{P}^{(n)}_{\bullet}:n\in \mathbb{N})$ of
coalescing Feller transition functions on $M$. It is assumed that
$(\mathsf{P}^{(n)}_{\bullet}:n\in \mathbb{N})$ satisfies conditions
(TF~\ref{item:consistency}), (TF~\ref{item:coalescing_property}),
(TF~\ref{item:continuity_of_trajectories_TF3}) and
(TF~\ref{item:14}). As in Section~\ref{sec:strong-flow-meas}, we let
$\mathcal{D}=\{(s_n,x_n):n\in \mathbb{N}\}$ be a dense subset of
\(\mathbb{R}\times{M}\) and let $\varPhi=(\varPhi_n: n\in \mathbb{N})$
be the skeleton of \(\psi^{0}\) constructed in
Section~\ref{sec:rand-sequ-varphi}. Recall that for each \(n\),
$\varPhi_n$ is a continuous modification of
$\psi^0_{s_n,\cdot}(x_n)$. Let also \(\theta\) be the measurable
continuous modification of $\psi^0$ defined in
Section~\ref{sec:meas-modif} such that \(\theta\) preserves
\(\varPhi\) a.s. In this section, it will be proves that if $\varPhi$
satisfies the ICP a.s., then $\theta$ is a strong measurable
continuous modification of $\psi^0$.

\subsection{A sufficient condition for the strong flow property}
\label{sec:sufficient-condition}

We give a simple condition under which $\theta$ is a strong measurable
continuous modification of \(\psi^{0}\). Note that if one takes
\(F=\mathbb{R}\times{M}\) in Theorem~\ref{thm:sfp_deterministic}, then
\(k^s_x=0\) for all \((s,x)\in\mathbb{R}\times{M}\) and
$\psi_{s,t}=\theta_{s,t}$ for all $s\leq t$.

\begin{lemma} \label{lem:icp_strong_flow} Assume that a.s., for any
  $s<t$ and $x\in M$
  \begin{equation}
    \label{eq:icp}
    \exists n\in I^{t}, \qquad \theta_{s,\cdot}(x)[t,\infty)=\varPhi_n[t,\infty).
  \end{equation}
  Then $\theta$ is a strong measurable continuous modification of $\psi^0$.
\end{lemma}

\begin{proof}
  Let \(\Omega^{0}\) be an event of probability one on which
  \(\varPhi\) is a skeleton, \(\theta\) preserves \(\varPhi\) and
  \eqref{eq:icp} is satisfied for all \(s<t\) and \(x\in M\). Let us
  fix \(\omega\in\Omega^{0}\). Let $s<t$ and $x\in M$. Then
  \eqref{eq:icp} with the fact that $\theta$ preserves the skeleton
  implies that for all \(u>t\),
  \begin{displaymath}
    \theta_{t,u}(\theta_{s,t}(x))
    = \theta_{t,u}(\varPhi_n(t))
    = \varPhi_n(u)
    = \theta_{s,u}(x).
  \end{displaymath}
  This proves the lemma.
\end{proof}

The condition given in Lemma \ref{lem:icp_strong_flow} means that
after any arbitrary short time, trajectories of $\theta$ meet
trajectories of $\varPhi$. This happens under the instantaneous
coalescence property we introduce in the next section.

\subsection{The instantaneous coalescence property}
\label{sec:inst_coalesc_general}

In this section we give a sufficient condition that allows to apply
Lemma~\ref{lem:icp_strong_flow} and therefore to prove that \(\theta\)
is a strong measurable continuous modification of \(\psi^{0}\).

\begin{definition}
  \label{def:skeleton_instant_coalescence}
  A sequence \(\varphi\in X^{\mathbb{N}}\) is said to possess the
  instantaneous coalescence property (ICP) if for any $s<t$, the set
  $\{\varphi_n(t):\ n\in I^s\}$ is locally finite.
\end{definition}

\begin{lemma}
  \label{lem:strong_flow_instant_coalescence}
  Assume that $\varPhi$ possesses the ICP almost surely. Then
  \(\theta\) is a strong measurable continuous modification of
  $\psi^0$.
\end{lemma}

\begin{proof} We just have to verify that the condition of
  Lemma~\ref{lem:icp_strong_flow} is satisfied by \(\theta\). Let us
  place on an event of probability one on which \(\varPhi\) is a
  skeleton possessing the ICP and \(\theta\) preserves \(\varPhi\).
  Let $t>s$, \(x\in{M}\) and $y=\theta_{s,t}(x)$.  The set
  \(\{\varPhi_{k}^{s,x}\}_{k\geq{1}}\) being relative compact, there
  is a subsequence $\{\varPhi^{s,x}_{k_i}\}_{i\geq 1}$ such that
  $\lim_{i\to\infty}\varPhi^{s,x}_{k_i}=\theta_{s,\cdot}(x)$ and in
  particular that $\lim_{i\to\infty}\varPhi^{s,x}_{k_i}(t)=y$.  The
  ball $B(y,\varepsilon)$ is relatively compact. The ICP implies that
  \(B(y,\varepsilon)\cap \{\varPhi_n(t):n\in I^{s}\}\) is a finite set
  and so there is \(i_{0}\) such that for all \(i\geq{i_{0}}\),
  \(\varPhi_{k_{i}}^{s,x}(t)=y\).  Hence, there exists $n\in I^s$ such
  that for all $i\geq{i_{0}}$,
  $\varPhi^{s,x}_{k_i}[t,\infty)=\varPhi_n[t,\infty)$. It follows that
  $$
  \theta_{s,u}(x)=\varPhi_n(u), \ u\geq t.
  $$
  The condition of Lemma~\ref{lem:icp_strong_flow} is verified and, as
  a consequence, \(\theta\) is a strong measurable continuous
  modification of $\psi^0$.
\end{proof}

\subsection{Sufficient condition ensuring the a.s. ICP for \(\varPhi\)}
\label{sec:suff-cond-ensur}
The sufficient condition given in this section is an extension of the one
given in \cite{zbMATH06176802}. We recall that \((X,Y)\) denotes the
canonical process on \(C([0,\infty):M^{2})\), so that under
\(\mathbb{P}^{(2)}_{(x,y)}\) the process \((X,Y)\) is a Markov process
with transition function $\mathsf{P}^{(2)}_{\bullet}$ and starting point
$(x,y).$ Also, $X^{(n)}=(X^1,\ldots,X^n)$ denotes the canonical
process on \(C([0,\infty):M^{2})\), so that under
\(\mathbb{P}^{(n)}_{x}\) the process \(X^{(n)}\) is a Markov process
with transition function $\mathsf{P}^{(n)}_{\bullet}$ and starting point $x.$

Let us say that a compact set \(K\subset{M}\) satisfies
\(\mathcal{P}\) if there are
positive and finite constants \(\alpha,\beta,\kappa,p\) and \(C\),
such that \(\alpha\kappa>1\) and such that conditions~(P\ref{item:22}) and
(P\ref{item:23}) given below are satisfied:
\begin{enumerate}[(P1)]
\item\label{item:22} For all \(\varepsilon>0, (x,y)\in{}K^{2}\), it holds
  that
  \begin{equation}
    \label{eq:3}
    \rho(x,y)\leq \varepsilon\: \Longrightarrow{}\:
    \mathbb{P}^{(2)}_{(x,y)}[\sigma\wedge\tau>\beta\varepsilon^{\alpha}]\leq{1-p}
  \end{equation}
  where \(\sigma=\inf\{t:(X_{t},Y_{t})\not\in{K^{2}}\}\) and
  \(\tau=\inf\{t:X_{t}=Y_{t}\}\);
\item\label{item:23} For all \(A\subset{K}\) and \(n\geq{1}\), it holds that
  \begin{equation}
    \label{eq:4}
    \#A\geq{}n
    \:\Longrightarrow\:
    \rho(x,y)\leq{}Cn^{-\kappa}\:
    \textrm{for some}\:(x,y)\in{}A^{2}, \ x\ne y.
  \end{equation}
\end{enumerate}

We will use the following notation in this section: for \(s\leq{t}\) and
\(n\in{I^{s}}\), the point \(\varPhi_{n}(t)\) will be denoted
\(\varphi_{s,t}(x)\) where \(x=\varPhi_{n}(s)\). Then
\(\varphi_{s,\cdot}:[s,\infty)\to{M},\:t\mapsto{}\varphi_{s,t}(x)\) is
continuous (i.e. \(\varphi_{s,\cdot}\in{}C([s,\infty):M)\)) and
\(\varphi_{s,s}(x)=x\).

Let us first prove a localized ICP.  For all \(s\leq{}t\) and compact
\(K\subset M\), set
\(A^{K}_{s,t}=\{\varPhi_{n}(t):n\in{I^{s}},\:\varPhi_{n}([s,t])\subset{K}\}\).
\begin{proposition}
  \label{prop:A_t_finite}
  If \(K\subset{M}\) is a compact set satisfying \(\mathcal{P}\), then
  a.s. \(A^{K}_{s,t}\) is a finite set for all \(t>s\).
\end{proposition}

For the following
lemmas~\ref{lem:EMS_1},~\ref{lem:EMS_2},~\ref{lem:EMS_3} and in the
proof of Proposition~\ref{prop:A_t_finite}, we let \(K\) be a compact
subset of \(M\) satisfying \(\mathcal{P}\) and let
\(\alpha,\beta,\kappa,p\) and \(C\) be constants such that
\(\alpha\kappa>1\) and such that~(P\ref{item:22}) and~(P\ref{item:23})
are satisfied.  For \(n\geq{1}\), let
\(\sigma^{n}=\inf\{t:X^{(n)}_{t}\not\in{K^{n}}\}\) and for
\(m\leq{}n\), let \(\tau_{m}^{n}\) be the first time \(t\) such that
\(\#\{X^{1}_{t},\ldots,X^{n}_{t}\}\leq{}m\).

  \begin{lemma}\label{lem:EMS_1}
    Let \(x\in{K^{n}}\) and \(k\geq{1}\) be such that
    \(\#\{x^{1},\ldots,x^{n}\}\leq{k}\). Then for all \(t\geq 0\) and
    \(\varepsilon={Ck^{-\kappa}}\)
    \begin{equation}\label{eq:5}
      \mathbb{P}^{(n)}_{x}[\sigma^{n}\wedge\tau^{n}_{k-1}>\beta\varepsilon^{\alpha}+t]
      \leq (1-p) \mathbb{P}^{(n)}_{x}[\sigma^{n}\wedge\tau^{n}_{k-1}>t]
    \end{equation}
  \end{lemma}
  \begin{proof}
    Let \(x\in{K^{n}}\) and \(k\geq{1}\). If
    \(\#\{x^{1},\ldots,x^{n}\}\leq{k-1}\) then \(\tau^{n}_{k-1}=0\) and
    \eqref{eq:5} holds. If \(\#\{x^{1},\ldots,x^{n}\}={k}\), then there
    is \(i\neq{j}\) such that \(\rho(x^{i},x^{j})\leq {Ck^{-\kappa}}\).
    Therefore, since \(\varepsilon=Ck^{-\kappa}\),
    \begin{equation*}
      \mathbb{P}^{(n)}_{x}[\sigma^{n}\wedge\tau^{n}_{k-1}>\beta\varepsilon^{\alpha}]
      \leq
      \mathbb{P}^{(2)}_{(x^{i},x^{j})}[\sigma\wedge\tau>\beta\varepsilon^{\alpha}]
      \leq 1-p.
    \end{equation*}
    Using the Markov property at time \(t\), we have
    \begin{align*}
      \mathbb{P}^{(n)}_{x}[\sigma^{n}\wedge\tau^{n}_{k-1}>\beta\varepsilon^{\alpha}+t]
      & =
        \mathbb{E}^{n}_{x}
        \left[
        1_{{\{\sigma^{n}\wedge\tau^{n}_{k-1}>{t}\}}}
        \mathbb{P}^{(n)}_{X^{n}_{t}}[\sigma^{n}\wedge\tau^{n}_{k-1}>\beta\varepsilon^{\alpha}]
        \right]\\
      & \leq
        (1-p)\mathbb{P}^{(n)}_{x}[\sigma^{n}\wedge\tau^{n}_{k-1}>{}t]
    \end{align*}
    Where we have used in the last inequality the fact that, on the
    event \(\{\sigma^{n}\wedge\tau^{n}_{k-1}>{t}\}\),
    \(X^{(n)}_{t}\in{K^{n}}\) and
    \(\#\{X^{1}_{t},\ldots,X^{n}_{t}\}={k}\).
  \end{proof}

  \begin{lemma}\label{lem:EMS_2}
    Let \(x\in{K^{n}}\) and \(k\geq{1}\). Then, if
    \(\#\{x^{1},\ldots,x^{n}\}\leq{k}\),
    \begin{enumerate}[(i)]
    \item\label{item:5} for all \(j\geq{1}\),
      \(\mathbb{P}^{(n)}_{x}[\sigma^{n}\wedge\tau^{n}_{k-1}>{j}\beta\varepsilon^{\alpha}]
      \leq (1-p)^{j}\), where \(\varepsilon={Ck^{-\kappa}}\);
    \item\label{item:6}
      \(\mathbb{E}^{(n)}_{x}[\sigma^{n}\wedge\tau^{n}_{k-1}]\leq\frac{C_{1}}{k^{\kappa\alpha}}\),
      where \(C_{1}=\frac{\beta{C^{\alpha}}}{p}\).
    \end{enumerate}
  \end{lemma}
  \begin{proof}
    Applying Lemma~\ref{lem:EMS_1} \(j\) times, one easily obtains
    (\ref{item:5}). Now, using \eqref{item:5}, we obtain
    \begin{align*}
      \mathbb{E}^{(n)}_{x}[\sigma^{n}\wedge\tau^{n}_{k-1}]
      & = \sum_{j=0}^{\infty}
        \int_{j\beta\varepsilon^{\alpha}}^{(j+1)\beta\varepsilon^{\alpha}}
        \mathbb{P}^{(n)}_{x}[\sigma^{n}\wedge\tau^{n}_{k-1}>{t}]\,\mathrm{d}t
      \\
      & \leq \sum_{j=0}^{\infty}
        \beta\varepsilon^{\alpha}
        \times
        \mathbb{P}^{(n)}_{x}[\sigma^{n}\wedge\tau^{n}_{k-1}>{j\beta\varepsilon^{\alpha}}]
      \\
      & \leq \sum_{j=0}^{\infty}
        \beta\varepsilon^{\alpha}\times(1-p)^{j}
        \leq \frac{\beta{C}^{\alpha}}{pk^{\kappa\alpha}}.
    \end{align*}
    This proves \eqref{item:6}.
  \end{proof}

\begin{lemma}\label{lem:EMS_3}
  Let \(x\in{K^{n}}\) and \({1}\leq {m}\leq{n}\). Then
  \begin{equation}
    \mathbb{E}^{(n)}_{x}[\sigma^{n}\wedge\tau^{n}_{m}]\leq\frac{C_{2}}{m^{\kappa\alpha-1}},
  \end{equation}
  where \(C_{2}=\frac{C_{1}}{\kappa\alpha-1}\).
\end{lemma}
\begin{proof}
  Since
  \(\sigma^n\wedge
  \tau^{n}_{m}=\sum_{k=m+1}^{n}(\sigma^n\wedge\tau^{n}_{k-1}-\sigma^n\wedge\tau^{n}_{k})\),
  using the strong Markov property and
  Lemma~\ref{lem:EMS_2}-\eqref{item:6}, we obtain
  \begin{align*}
    \mathbb{E}^{(n)}_{x}[\sigma^{n}\wedge\tau^{n}_{m}]
    &=
      \sum_{k=m+1}^{n}
      \mathbb{E}^{(n)}_{x}[\sigma^{n}\wedge\tau^{n}_{k-1}-\sigma^{n}\wedge\tau^{n}_{k}]
    \\
    &=
      \sum_{k=m+1}^{n}
      \mathbb{E}^{(n)}_{x}\left[
      \mathbb{E}^{(n)}_{X^{(n)}_{\sigma^{n}\wedge\tau^{n}_{k}}}[\sigma^{n}\wedge\tau^{n}_{k-1}]
      \right]
    \\
    &\leq
      \sum_{k=m+1}^{\infty}  \frac{C_{1}}{k^{\kappa\alpha}}
      \leq \frac{C_{1}}{(\kappa\alpha-1)m^{\kappa\alpha-1}}.
  \end{align*}
  The lemma is proved.
\end{proof}

\begin{proof}[Proof of Proposition~\ref{prop:A_t_finite}]
  We suppose at first that \(s=0\) and set \(A_{t}=A^{K}_{0,t}\) for
  \(t\geq{0}\).
  
  Given
  \(\mathcal{A}_{0}=\sigma(\varPhi_{n}(r):n\in\mathbb{N},s_{n}\leq{r}\leq{0})\),
  the set \(A_{0}\) is at most countable. For \(n\geq{1}\), set
  \(\mathcal{P}_{n}(A_{0})\) the set of all subsets of \(A_{0}\) with
  cardinality \(n\). If
  \(\{x^{1},\ldots{},x^{n}\}\in\mathcal{P}_{n}(A_{0})\), we have that
  \((\varphi_{0,\cdot}(x^{1}),\ldots{},\varphi_{0,\cdot}(x^{n}))\) is a
  Markov process of law
  \(\mathbb{P}^{(n)}_{(x^{1},\ldots{},x^{n})}\). Then, for all \(t>0\),
  \begin{align*}
    \mathbb{P}[\#A_{t}=\infty|\mathcal{A}_{0}]
    &=
      \lim_{m\to\infty}\mathbb{P}[\#A_{t}\geq{m+1}|\mathcal{A}_{0}]
    \\
    &\leq 
      \lim_{m\to\infty}
      \lim_{n\to\infty}
      \sup_{\{x^{1},\ldots{},x^{n}\}\in\mathcal{P}_{n}(A_{0})}
      \mathbb{P}^{(n)}_{(x^{1},\ldots{},x^{n})}[\sigma^{n}\wedge\tau^{n}_{m}>t]
    \\
    &\leq
      \lim_{m\to\infty}
      \lim_{n\to\infty}
      \sup_{\{x^{1},\ldots{},x^{n}\}\in\mathcal{P}_{n}(A_{0})}
      \frac{1}{t}\times\mathbb{E}^{(n)}_{(x^{1},\ldots{},x^{n})}[\sigma^{n}\wedge\tau^{n}_{m}]
    \\
    &\leq
      \lim_{m\to\infty}
      \frac{C_{2}}{tm^{{\kappa\alpha-1}}}
      = 0
  \end{align*}
  Therefore \(\mathbb{P}[\#A_{t}=\infty]=0\) and this implies that
  (using that \(\{\#A_{t}=\infty\}\subset\{\#A_{r}=\infty\}\) for all
  \(r<t\)) a.s. for all \(t>0\), \(A_{t}\) is a finite set.

  We have thus proved that for all \(s\in\mathbb{R}\), a.s. for all
  \(t>s\), \(A^{K}_{s,t}\) is a finite set. This implies that a.s. for
  all \(s\in\mathbb{Q}\) and \(t>s\), \(A^{K}_{s,t}\) is a finite
  set. Let \(E\) be an event on which there is \(s<t\) such that
  \(\#A^{K}_{s,t}=\infty\). Then on \(E\), there is \(a\in\mathbb{Q}\)
  such that \(s<a<t\) and such that \(\#A^{K}_{a,t}=\infty\). This
  implies that
  \(\mathbb{P}[E]\leq\sum_{a\in\mathbb{Q}}\mathbb{P}[\exists{t>a},\#A^{K}_{a,t}=\infty]=0\).
\end{proof}

\begin{remark}
  When  the space \(M\) is compact, Proposition~\ref{prop:A_t_finite} shows
  that if \(M\) satisfies \(\mathcal{P}\), then
  a.s. for all \(s<t\), \(\{\varPhi_{n}(t):n\in{I^{s}}\}\) is a finite
  set, and \(\varPhi\) possesses the ICP.
\end{remark}

When \(M\) is locally compact, then it is not clear that
Proposition~\ref{prop:A_t_finite} implies that \(\varPhi\) possesses the
ICP a.s. To prove the ICP, we need an additional assumption.

\begin{proposition}
  \label{prop:icp-when-m}
  Let \((K_{l}: l\in\mathbb{N})\) be an increasing sequence of compact
  sets such that \(\cup_{l\in\mathbb{N}}K_{l}=M\).  Suppose that
  \begin{enumerate}[(i)]
  \item\label{item:24} for any \(l\in\mathbb{N}\), \(K_{l}\) satisfies
    \(\mathcal{P}\);
  \item\label{item:25} a.s. for all \(k\in\mathbb{N}\) and \(s<t\),
    there are \(l>k\) and \(r\in(s,t)\) such that for all
    \(n\in\mathbb{N}\),
    \begin{equation} \label{eq:6}
      s_{n}\leq{s}\quad\mathrm{and}\quad\varPhi_{n}(t)\in{K_{k}}
      \quad\Longrightarrow\quad
      \varPhi_{n}([r,t])\subset{}K_{l}.
    \end{equation}
  \end{enumerate}
  Then a.s. \(\varPhi\) possesses the ICP.
\end{proposition}
\begin{proof}
  For \(l\in\mathbb{N}\) and \(s<t\), we set
  \(A^{l}_{s,t}=A^{K_{l}}_{s,t}\) and
  \(B^{l}_{s,t}=\{\varPhi_{n}(t):n\in{I^{s}}\}\cap{K_{l}}\).

  Let \(\Omega_{0}\) be an event of probability one on which for all
  \(k\in\mathbb{N}\) and \(s<t\), it holds that \(A^{k}_{s,t}\) is a
  finite set and that there are \(l>k\) and \(r\in(s,t)\) such that
  \eqref{eq:6} is satisfied for all \(n\in\mathbb{N}\).
  On \(\Omega_{0}\),
  \begin{align*}
    B^{k}_{s,t}
    &\subset
      \{\varPhi_{n}(t):s_{n}\leq{s}\:\mathrm{and}\:\varPhi_{n}([r,t])\subset{K_{l}}\}
    \\
    &\subset
      \{\varPhi_{n}(t):s_{n}\leq{r}\:\mathrm{and}\:\varPhi_{n}([r,t])\subset{K_{l}}\}
      = A^{l}_{r,t}
  \end{align*}
  which is a finite set. Therefore on \(\Omega_{0}\), for all \(s<t\)
  and \(k\in\mathbb{N}\), \(B^{k}_{s,t}\) is a finite set. This implies
  that on \(\Omega_{0}\), 
  \(\{\varPhi_{n}(t):n\in{I^{s}}\}\) is locally finite for all \(s<t\).
\end{proof}

\section{Stochastic flows on metric graphs}
\label{sec:sfp_IC}

Let $(M,\rho)$ be a metric graph (metric graphs are defined in
Section~\ref{sec:notations}).  Let \(\psi^{0}\) be a stochastic flow
of measurable mappings in $M$ associated to a consistent sequence of
coalescing Feller transition functions
$(\mathsf{P}^{(n)}_{\bullet}:n\in \mathbb{N})$ on $M$.  It is assumed that
$(\mathsf{P}^{(n)}_{\bullet}:n\in \mathbb{N})$ satisfies conditions
(TF~\ref{item:consistency}), (TF~\ref{item:coalescing_property}) and
(TF~\ref{item:continuity_of_trajectories_TF3}). As in
Section~\ref{sec:strong-flow-meas}, we let
$\mathcal{D}=\{(s_n,x_n):n\in \mathbb{N}\}$ be a dense subset of
\(\mathbb{R}\times{M}\) and let $\varPhi=(\varPhi_n: n\in \mathbb{N})$
be the random sequence constructed out of \(\psi^{0}\) in
Section~\ref{sec:rand-sequ-varphi}. Then $\varPhi$
a.s. satisfies conditions (Sk~\ref{item:coalescence}) (by
(TF~\ref{item:coalescing_property})) and (Sk~\ref{item:denseness}) (by
construction). In
subsection~\ref{sec:inst_coalescence_on_metric_graph}, we show that if
\(\varPhi\) possesses the ICP a.s. then \(\varPhi\) a.s. satisfies
(Sk~\ref{item:rel_comp}), and a.s. is a skeleton. In
subsection~\ref{sec:suff-cond-icp} a sufficient condition on $\varPhi$
to satisfy the ICP a.s. is given.

\subsection{Notations}
\label{sec:notations}
A locally compact separable metric space $(M,\rho)$ is said to be a
metric graph if there are a countable set (of vertices) $V\subset M$
and a partition $\{E_j\}_{j\in J}$ of $M\setminus V$ (into edges) such
that
\begin{itemize}
\item For each \(j\in J\),
  \begin{itemize}
  \item the edge $E_j$ is an open subset of $M$;
  \item there is an isometry $e_j:(0,L_j)\to E_j$, $0<L_j\leq \infty$;
  \end{itemize}
\item for each $v\in V$,
  \begin{itemize}
  \item the set $J(v)=\{j\in J: v\in\partial E_j\}$ is non-empty and
    finite;
  \item the set $\{v\}\cup\bigcup_{j\in J(v)}E_j$ is a neighborhood of
    $v$.
  \end{itemize}
\end{itemize}
For every \(j\in{}J\), the isometry \(e_{j}\) can be extended by
continuity at \(0\) and at \(L_{j}\) (when \(L_{j}<\infty\)). By abuse
of notation, this extension will also be denoted \(e_{j}\). Then
\(e_{j}(0)\in{}V\) and \(e_{j}(L_{j})\in{}V\) when \(L_{j}<\infty\).

For a vertex \(v\in{V}\), the cardinality of $J(v)$ is denoted by $d(v)$
and is called the degree of \(V\). For a point $x\in{}M\setminus{}V$,
the degree of \(x\) is defined by \(d(x)=2\). In this case, $x=e_j(t)$
for some \(j\in{J}\) and \(t\in(0,L_{j})\), and the sets
$e^{-1}_j((0,t))$ and $e^{-1}_j((t,L_j))$ will be viewed as the two
adjacent edges to $x$.

Local compactness of the space \((M,\rho)\) implies \(d(v)<\infty\) for
every \(v\in{}V\).  Without loss of generality we will assume that $M$ has
no loops, i.e. if $L_j<\infty$ then $\partial E_j$ contains exactly two
vertices\footnote{When there is a
  loop one can add a vertex of degree $2$ on it.}.

When \(x\) and \(y\) are two points belonging to the same edge
\(\overline{E_{j}}\), then there is \((r,s)\in[0,L_{j}]^{2}\) such that
\(x=e_{j}(r)\) and \(y=e_{j}(s)\). When \(r\leq{}s\) will denote by
\([x,y]\) (respectively \([x,y)\), \((x,y]\) and \((x,y)\)) the set
\(e_{j}^{-1}[r,s]\) (respectively \(e_{j}^{-1}[r,s)\),
\(e_{j}^{-1}(r,s]\) and \(e_{j}^{-1}(r,s)\)). When \(r>s\), we will
denote by \([x,y]\) (respectively \([x,y)\), \((x,y]\) and \((x,y)\))
the set \([y,x]\) (respectively \((y,x]\), \([y,x)\) and \((y,x)\)).

A bounded connected open set \(\mathcal{O}\subset{M}\) that contains at
most one vertex \(v\in{V}\) and whose boundary contains no vertices
\(v\in{V}\) (\(\partial\mathcal{O}\cap{V}=\emptyset\)) will be called a
simple open set. The closure of a simple open set is compact and will be
called a simple compact set.

Let \(\mathcal{O}\) be a simple open set.  There are two possibilities:
\begin{itemize}
\item \(V\cap\mathcal{O}=\emptyset\) and there are \(j\in{J}\) and
  \((u_{1},u_{2})\in{E_{j}^{2}}\) such that
  \(\mathcal{O}=(u_{1},u_{2})\). In this case \(\mathcal{O}\) will be
  called a simple neighborhood of $x$ for any $x\in \mathcal{O}$ and
  will be denoted by $U_{x}\big(\{u_{1},u_{2}\}\big)$.
\item \({V}\cap\mathcal{O}=\{v\}\) and for any \(j\in{J(v)}\), there is
  \(u_{j}\in{E_{j}}\) such that
  \begin{displaymath}
    \mathcal{O}=\cup_{j\in{}J(v)}[v,u_{j}).
  \end{displaymath}
  In this case $\mathcal{O}$ will be called a simple neighborhood of $v$
  will be denoted by \(U_{v}\big(\{u_{j}:1\leq{j}\leq{d(v)}\}\big)\).
\end{itemize}
Note that for any $x\in M$ and all sufficient small $\varepsilon>0$,
the ball $B(x,\varepsilon)$ is a simple neighborhood of $x$. For all
\(x\in M\), we denote by \(E_{1}^{x},\ldots,E_{d(x)}^{x}\) the edges
adjacent to \(x\).

\subsection{Instantaneous coalescence property on metric graphs}
\label{sec:inst_coalescence_on_metric_graph}

\begin{theorem}
  \label{thm:inst_coalescence_metric_graph}
  On a metric graph, if \(\varPhi\) possesses the ICP a.s., then \(\varPhi\) is
  a skeleton a.s.
\end{theorem}

Before proving this theorem, let us prove the following simple
lemma. 
\begin{lemma}\label{lem:inst-coal-prop}
  Almost surely, \(\varPhi\) satisfies the following property:
  \begin{itemize}
  \item\label{item:27} For all \(x\in{M}\), \(\varepsilon>0\) and \(s\in\mathbb{R}\),
    with \(\varepsilon<\rho(x,V\setminus{}\{x\})\), there are
    \(\delta=\delta_{s,x}^{\varepsilon}>0\) and
    \(\{n_{1},\ldots{},n_{d(x)}\}\subset{}I^{s-\delta}\) such that
    \begin{enumerate}[(i)]
    \item\label{item:7}
      \(\varPhi_{n_{j}}([s-\delta,s+\delta])\subset{}E^{x}_{j}\cap
      B(x,\varepsilon)\setminus B(x,3\varepsilon/4)\) for all
      \(j\in\{1,\ldots,d(x)\}\).
    \item\label{item:8} For all \(t\in[s-\delta,s+\delta]\),
      \(U_{s,x}^{\varepsilon}(t):=U_{x}\big(\{\varPhi_{n_{1}}(t),\ldots,\varPhi_{n_{d(x)}}(t)\}\big)\)
      is a simple neighborhood of \(x\), such that
      \(B(x,\frac{3\varepsilon}{4})\subset{}U_{s,x}^{\varepsilon}(t)\subset{}B(x,\varepsilon)\).
    \end{enumerate}
  \end{itemize}
\end{lemma}

\begin{proof}
  Let \(\Omega^{0}\) be an event of probablility one on which
  \(\varPhi\) satisfies (SK~\ref{item:coalescence}),
  (SK~\ref{item:denseness}) and on which it holds that for all
  \(t\in\mathbb{R}\) the set \(\{\varPhi_{n}(t):s_{n}<t\}\) is dense
  in \(M\). Lemma~\ref{lem:denseness} implies that such event
  exists. We now check that the property claimed in the lemma is
  satisfied by \(\varPhi\) on \(\Omega^{0}\). Let us fix \(x\in{M}\),
  \(\varepsilon>0\) and \(s\in\mathbb{R}\), with
  \(\varepsilon<\rho(x,V\setminus\{x\})\).

  Let us also fix \(j\in\{1,\ldots,d(x)\}\) and let \(y_{j}\in E_{j}^{x}\) be
  such that \(\rho(y_{j},x)=\frac{7\varepsilon}{8}\). Then there is \(n_{j}\)
  such that \(s_{n_{j}}<s\) and
  \(\rho(\varPhi_{n_{j}}(s),y_{j})<\frac{\varepsilon}{8}\). The mapping
  \(\varPhi_{n_{j}}\) being continuous, there is \(\delta_{j}>0\) such that
  \(s_{n_{j}}<s-\delta_{j}\) and such that
  \(\varPhi_{n_{j}}([s-\delta_{j},s+\delta_{j}])\subset{}B(y_{j},\frac{\varepsilon}{8})\subset
  E_{j}^{x}\cap B(x,\varepsilon)\setminus
  B(x,3\varepsilon/4)\). Taking \(\delta=\inf_{j}\delta_{j}\),
  this proves~\eqref{item:7}, and~\eqref{item:8} is a straightforward
  consequence of~\eqref{item:7}.
\end{proof}

\begin{proof}[Proof of Theorem~\ref{thm:inst_coalescence_metric_graph}]
  Let us place on an event \(\Omega^{0}\) of probability one on which
  \(\varPhi\) possesses the ICP and
  satisfies~(Sk~\ref{item:coalescence}),~(Sk~\ref{item:denseness}) and
  the property given in Lemma~\ref{lem:inst-coal-prop}. In order to
  prove that \(\varPhi\) is a skeleton on \(\Omega^{0}\), it remains
  to verify that for any compact $L\subset \mathbb{R}\times M$, the
  set
  $\mathcal{F}_L=\{\varPhi_n[s,\infty): n\in I^s, (s,\varPhi_n(s))\in
  L\}$ is relatively compact in $X$. This will be done using
  Theorem~\ref{lem:ascoli-arzela} given below, which is a version of the
  Ascoli-Arzela theorem for the space $X$.

  Let $L\subset \mathbb{R}\times M$ be a compact set. 
  Item~(\ref{item:3}) of Theorem~\ref{lem:ascoli-arzela} is satisfied by
  \(\mathcal{F}_{L}\) since for all \(f\in\mathcal{F}_{L}\), there is
  \((n,s)\) such that \(n\in{I^{s}}\), \((s,\varPhi_{n}(s))\in{L}\) and
  \(f=\varPhi_{n}[s,\infty)\), which implies that
  \((i(f),f(i(f)))=(s,\varphi_{n}(s))\in{L}\), and \(L\) is compact.

  We now verify that item~(\ref{item:4}) of
  Theorem~\ref{lem:ascoli-arzela} is satisfied by \(\mathcal{F}_{L}\)
  for all \(\omega\in\Omega_{0}\).  We fix \(\varepsilon>0\) and let
  \(C<\infty\) be such that \(s\leq{C}\) for all \((s,x)\in{L}\).  For
  any $x\in M$, set
  $\varepsilon(x)=\varepsilon\wedge\frac{1}{2}\rho(x,V\setminus
  \{x\})$. Lemma~\ref{lem:inst-coal-prop} implies that on
  \(\Omega_{0}\), for any \((s,x)\in{}\mathbb{R}\times{}M\) there exist
  $\delta^{\varepsilon(x)}_{s,x}>0$ and a simple neighborhood
  \(U^{\varepsilon(x)}_{s,x}(t)\) of \(x\) such that
  \(B(x,\frac{3\varepsilon(x)}{4})\subset{}U_{s,x}^{\varepsilon(x)}(t)\subset{}B(x,\varepsilon(x))\)
  for all
  \(t\in[s-\delta_{s,x}^{\varepsilon(x)},s+\delta_{s,x}^{\varepsilon(x)}]\). The
  set
  \begin{equation*}
    W_{s,x}^{\varepsilon(x)}=
    \left[s-\frac{\delta_{s,x}^{\varepsilon(x)}}{2},s+\frac{\delta^{\varepsilon(x)}_{s,x}}{2}\right]\times
    B\left(x,\frac{3\varepsilon(x)}{4}\right)
  \end{equation*}
  is a neighborhood of $(s,x)\in\mathbb{R}\times{M}$ and for all
  \((t,y)\in{}W_{s,x}^{\varepsilon(x)}\), we have that
  \(y\in{}U^{\varepsilon(x)}_{s,x}(t)\).

  The set \(L\) being compact, there exists a finite covering of \(L\):
  $L\subset\cup^m_{k=1}W_{s_{k},x_{k}}^{\varepsilon(x_{k})}.$ By the
  ICP, for each $k\in \{1,\ldots,m\}$, setting
  \(\delta_{k}=\delta^{\varepsilon(x_k)}_{s_{k},x_{k}}\), the set
  $$
  \left\{\varPhi_n(s_{k}+\delta_{k}): n\in
    I^{s_k+\frac{\delta_{k}}{2}}\right\}\cap B(x_{k},\varepsilon)
  $$
  is finite. So, for every \(k\), there is a finite family
  \(\{n_{k}^{l}:1\leq{l}\leq{N_{k}}\}\subset{}I^{s_{k}+\frac{\delta_{k}}{2}}\)
  such that
  $$
  \left\{\varPhi_n(s_k+\delta_{k}):n\in{}I^{s_k+\frac{\delta_{k}}{2}}\right\}\cap{}B(x_k,\varepsilon)=
  \{\varPhi_{n_{k}^{l}}(s_k+\delta_{k}):1\leq{l}\leq{N_{k}}\}.
  $$
  We will set \(s_{k}^{l}=i(\varPhi_{n_{k}^{l}})\) and
  \(C_{k}=C\vee{(s_{k}+\delta_{k})}\). Then
  \(s_{k}^{l}\leq{s_{k}+\frac{\delta_{k}}{2}}\leq{C_{k}}\). The family
  of mappings \(\{\varPhi_{n_{k}^{l}}\}\) being finite, this family is
  uniformly equicontinuous and so there exists \(\alpha>0\) such that
  for all \((k,l)\) and \((r_{1},r_{2})\in[s_{k}^{l},C_{k}]^{2}\) with
  \(|r_{2}-r_{1}|<\alpha\) , we have
  \(\rho(\varPhi_{n_{k}^{l}}(r_{1}),\varPhi_{n_{k}^{l}}(r_{2}))<\varepsilon\).

  Let now \(s\in\mathbb{R}\) and $n\in I^s$ be such that
  $(s,\varPhi_n(s))\in L$. Then
  $(s,\varPhi_n(s))\in W^{\varepsilon(x_k)}_{s_k,x_k}$ for some
  $k\in\{1,\ldots,m\}$. It follows that
  $|s-s_k|\leq\frac{\delta_{k}}{2}$ and
  $\varPhi_n(s)\in U^{\varepsilon}_{s_k,x_k}(s)$. Then
  \(n\in{}I^{s_{k}+\frac{\delta_{k}}{2}}\).
  \begin{itemize}
  \item Using (Sk~\ref{item:coalescence}) and the fact that
    \(U^{\varepsilon(x_k)}_{s_{k},x_{k}}(t)\) is a simple neighborhood
    of \(x_{k}\) for all \(t\in[s,s_{k}+\delta_{k}]\), we have that for
    all $t\in [s,s_k+\delta_{k}]$,
    \begin{equation*}
      \varPhi_n(t)\in U^{\varepsilon(x_k)}_{s_k,x_k}(t)\subset B(x_k,\varepsilon(x_k))\subset B(x_k,\varepsilon).
    \end{equation*}
    This implies that for all
    \((r_{1},r_{2})\in[s,s_{k}+\delta_{k}]^{2}\),
    \begin{equation*}
      \rho(\varPhi_{n}(r_{1}),\varPhi_{n}(r_{2})\leq{2\varepsilon}.
    \end{equation*}
  \item Since \(n\in{}I^{s_{k}+\frac{\delta_{k}}{2}}\) and
    \(\varPhi_{n}(s_{k}+\delta_{k})\in{}B(x_{k},\varepsilon)\), we have
    that
    \(\varPhi_{n}(s_{k}+\delta_{k})=\varPhi_{n_{k}^{l}}(s_{k}+\delta_{k})\)
    for some \(l\). Hence, using (Sk~\ref{item:coalescence}),
    \begin{equation*}
      \varPhi_n[s_k+\delta_{k},\infty)=\varPhi_{n^k_l}[s_k+\delta_{k},\infty).
    \end{equation*}
    for some $l\in \{1,\ldots,N_{k}\}$. This implies that for all
    \((r_{1},r_{2})\in[s_{k}+\delta_{k},C_{k}]^{2}\) such that
    \(|r_{2}-r_{1}|<\alpha\),
    \begin{equation*}
      \rho(\varPhi_{n}(r_{1}),\varPhi_{n}(r_{2})\leq{\varepsilon}.
    \end{equation*}
  \end{itemize}
  Combining these two items, we obtain that for all
  \((r_{1},r_{2})\in[s,C]\) such that \(|r_{2}-r_{1}|<\alpha\),
  \begin{equation*}
    \rho(\varPhi_{n}(r_{1}),\varPhi_{n}(r_{2})\leq{3\varepsilon}.
  \end{equation*}
  Item~\eqref{item:4} of Theorem~\ref{lem:ascoli-arzela} is verified
  and, by the Ascoli-Arzela theorem, the family \(\mathcal{F}_{L}\) is
  relatively compact in $X$ for every compact \(L\) and the Theorem is
  proved.
\end{proof}

\begin{theorem}[Ascoli-Arzela theorem]\label{lem:ascoli-arzela}
  A family of mappings \(\mathcal{F}\subset{}X\) is relatively compact
  if and only if \(\mathcal{F}\) is uniformly equicontinuous, i.e. if
  and only if the following items are satisfied:
  \begin{enumerate}[(i)]
  \item \label{item:3} \(\{(i(f),f(i(f))):f\in\mathcal{F}\}\) is
    relatively compact;
  \item \label{item:4} For all \(\varepsilon>0\) and
    \(C\geq\sup\{i(f):f\in\mathcal{F}\}\), there is \(\alpha>0\) such
    that for all \(f\in\mathcal{F}\) and all
    \((r_{1},r_{2})\in[i(f),C]^{2}\) such that \(|r_{2}-r_{1}|<\alpha\),
    we have \(\rho(f(r_{1}),f(r_{2}))<\varepsilon\).
  \end{enumerate}
\end{theorem}

\subsection{Sufficient condition for the ICP on a metric graph}
\label{sec:suff-cond-icp}
In this section we give a sufficient condition on $\mathsf{P}^{(2)}_{\bullet}$,
the Feller transition function of the two-point motion of \(\psi^{0}\),
under which \(\varPhi\) possesses the ICP a.s.

Consider the following condition: 
\begin{enumerate}[(TF~1)]\setcounter{enumi}{4}
\item \label{item:tf_icp} For any simple compact $K\subset M$ there are
  constants $\beta,p>0$ such that for all \((x,y)\in {K}^{2}\) and
  \(\varepsilon>0\),
  \begin{equation}
    \label{eq:4_17}
    \rho(x,y)\leq \varepsilon\: \Longrightarrow{}\:
    \mathbb{P}^{(2)}_{(x,y)}[\sigma\wedge\tau>\beta\varepsilon^{2}]\leq{1-p}
  \end{equation}
  where \(\sigma=\inf\{t:(X_{t},Y_{t})\not\in{K}^{2}\}\) and
  \(\tau=\inf\{t:X_{t}=Y_{t}\}\);

\end{enumerate}

\begin{theorem}
  \label{thm:metric_graph_instantaneous_coalescence}
  Assume that the sequence $(\mathsf{P}^{(n)}_{\bullet}:n\in \mathbb{N})$ satisfies
  conditions (TF~\ref{item:consistency}),
  (TF~\ref{item:coalescing_property}),
  (TF~\ref{item:continuity_of_trajectories_TF3}) and
  (TF~\ref{item:tf_icp}). Then a.s. \(\varPhi\) is a skeleton that
  possesses the ICP.
\end{theorem}

\begin{proof} 
  Using Proposition \ref{prop:icp-when-m}, we will prove that
  a.s. $\varPhi$ possesses the ICP by verifying that simple compact sets
  satisfy \(\mathcal{P}\) given in
  Section~\ref{sec:suff-cond-ensur}, i.e. \eqref{eq:3} and
  \eqref{eq:4} are satisfied, and by verifying
  Proposition~\ref{prop:icp-when-m}-\eqref{item:24}. Let \(K\) be a
  simple compact and let \(\beta,p>0\) be the constants given in
  Assumption (TF~\ref{item:tf_icp}), taking \(\alpha=2\) we obtain
  that \eqref{eq:3} holds and so (P~\ref{item:22}) is satisfied. Let
  us remark that on a metric graph, for any simple compact \(K\) there
  is \(C>0\) such that for all \(A\subset{K}\) and \(n\geq{1}\), it
  holds that
  \begin{equation}
    \label{eq:8}
    \#A\geq{}n\:\Longrightarrow\:\rho(x,y)\leq{}Cn^{-1}\:\textrm{for some}\:(x,y)\in{}A^{2}\setminus{}\Delta.
  \end{equation}
  This implies that \eqref{eq:4} and so (P~\ref{item:23}) holds with
  $\kappa=1$. Since \(\alpha\kappa=2>1\), the simple compact \(K\)
  satisfies \(\mathcal{P}\).

  We now verify~Proposition~\ref{prop:icp-when-m}-\eqref{item:24}. Let
  \((U_{i,k})_{{i\geq{1}},\,{k\geq{1}}}\) be a sequence of simple open
  sets such that
  \begin{itemize}
  \item For all \(k\geq{1}\), \(M=\cup_{i\geq{1}}U_{i,k}\);
  \item For all \(i\geq{1}\) and \(k\geq{1}\),
    \(\bar{U}_{i,k}\subset{U}_{i,k+1}\).
  \end{itemize}
  The space \(M\) being a locally compact separable metric space, such
  sequence exists.

  For every \(k\geq{1}\), let
  \(K_{k}=\cup_{i=1}^{k}\bar{U}_{i,k}\). Then \(K_{k}\subset{K_{k+1}}\)
  and \(M=\cup_{k\geq{1}}K_{k}\). For each \(i,k\), the open set
  \(U_{i,k}\) is a simple neighborhood of some point \(x_{i}\in{M}\)
  that can be written in the form
  \(U_{i,k}=\cup_{j=1}^{d(x_{i})}[x_{i},u_{ij,k})\), with
  \(u_{ij,k}\in{E_{j}^{x_{i}}}\). Note that the points \(u_{ij,k}\) are
  such that \(u_{ij,k}\in(x_{i},u_{ij,k+1})\).

  Let us place on an event of probability one such that for all for all
  $t$ the set $\{\varPhi_n(t):s_n<t\}$ is dense in
  $M$. Lemma~\ref{lem:denseness} ensures that this event exists.  Let us
  fix \(k\geq{1}\) and \(s<t\). Then for all \(i\leq{k}\) and
  \(j\in\{1,\ldots{},d(x_{i})\}\), there exists \(n_{ij,k}\) such that
  \(s_{n_{ij,k}}<t\) and such that
  \(\varPhi_{n_{ij,k}}(t)\in(u_{ij,k},u_{ij,k+1})\). By continuity of
  the trajectories \(\varPhi_{n_{ij,k}}\), there exists \(r\in(s,t)\)
  such that for all \(i\leq{k}\) and \(j\in\{1,\ldots{},d(x_{i})\}\),
  \(\varPhi_{n_{ij,k}}([r,t])\subset(u_{ij,k},u_{ij,k+1})\subset{U_{ij,k+1}}\).

  Let now \(n\geq{1}\) be such that \(s_{n}\leq{s}\) and
  \(\varPhi_{n}(t)\in{K_{k}}\). Then, there is \(i\leq{k}\) such that
  \(\varPhi_{n}(t)\in\bar{U}_{i,k}\). Since \(s_{n}<r\), the
  coalescing property implies that for all \(t'\in[r,t]\),
  \(\varPhi_{n}(t')\in
  U_{x_{i}}(\{\varPhi_{n_{ij,k}}(t'):1\leq{j}\leq{d(x_{i})}\})
  \subset{U_{i,k+1}}\). This proves that
  \(\varPhi_{n}([r,t])\subset{K_{k+1}}\). Proposition~\ref{prop:icp-when-m}-\eqref{item:24}
  is verified.

  Applying Proposition~\ref{prop:icp-when-m}, we prove that \(\varPhi\)
  possesses the ICP a.s. Applying now
  Theorem~\ref{thm:inst_coalescence_metric_graph}, we can conclude that
  \(\varPhi\) is a skeleton a.s.    
\end{proof}

Using the results of Theorem~\ref{thm:measurable_modification} and Lemma
\ref{lem:strong_flow_instant_coalescence}, we get the following
corollary of Theorem~\ref{thm:metric_graph_instantaneous_coalescence}.

\begin{corollary}
\label{cor:sfp_metric_graph}
  Assume that the sequence $(\mathsf{P}^{(n)}_{\bullet}:n\in \mathbb{N})$ satisfies
  conditions (TF~\ref{item:consistency}),
  (TF~\ref{item:coalescing_property}),
  (TF~\ref{item:continuity_of_trajectories_TF3}) and
  (TF~\ref{item:tf_icp}). Then the family of mappings \(\theta\)
  defined by~\eqref{eq:2} is a strong measurable continuous
  modification of \(\psi^{0}\).
\end{corollary}

\section{Examples}
\label{sec:examples}

\subsection{Coalescing independent Walsh Brownian motions on a metric
  graph}
\label{subsec:Coalescing independent Walsh Brownian motions on a metric graph}

Let $M$ be a metric graph with vertices $V$, edges $(E_j:j\in J)$, and
isometries $e_j:(0,L_j)\to E_j,$ $0<L_j\leq \infty$ (see Section
\ref{sec:notations}). We equip $M$ with the shortest path distance
$\rho$. For all \(j\in J\), denote $g_j=e_j(0)$ and $d_j=e_j(L_j)$
(with \(d_{j}=\infty\) when \(L_{j}=\infty\)). Further, for any
$v\in V$ denote $J^+(v)=\{j\in J: g_j=v\}$ and
$J^-(v)=\{j\in J: d_j=v\},$ so that $J(v)=J^+(v)\cup J^-(v).$ Assume
that $\inf_{j\in J}L_j>0.$ To each $v\in V$ and $j\in J(v)$ we
associate a parameter $p_j(v)\in [0,1],$ such that
$\sum_{j\in J(v)}p_j(v)=1.$ Denote by $D$ the set of all continuous
functions $f:M\to \mathbb{R},$ such that for every $j\in J,$
$f\circ e_j\in C^2({(0,L_j)})$ with bounded first and second
derivatives, and
$$
\sum_{j\in J^+(v)}p_j(v)(f\circ e_j)'(0+) =
\sum_{j\in J^-(v)}p_j(v)(f\circ e_j)'(L_j-).
$$
For $f\in D$ and $x=e_j(t)$ set $f'(x)=(f\circ e_j)'(t),$
$f''(x)=(f\circ e_j)''(t),$ and for $v\in V$ set $f'(v)=f''(v)=0.$ Let
$$
Af=\frac{1}{2}f'', \ f\in D.
$$ 
The operator $A$ generates a continuous Feller Markov process on $M$
(see, \cite{zbMATH00528639}). We will call such process the Walsh
Brownian motion (WBM) on $M$ with transmission parameters
\(p_{j}(v)\), \(v\in{V}\), \(j\in{J(v)}\).  We are interested in the
existence of a strong measurable modification of a stochastic flow of
measurable mappings in $M$, whose trajectories are WBM's that are
independent before meeting and coalesce at the meeting time. Let
$(\mathsf{P}_t:t\geq 0)$ be the transition function of a WBM on $M$
with transmission parameters \(p_{j}(v)\), \(v\in{V}\),
\(j\in{J(v)}\). We define $(\mathsf{P}^{(n)}_{\bullet}: n\in \mathbb{N})$ to be
a unique consistent sequence of coalescing Feller transition functions
on $M$ obtained from
$\left(\mathsf{P}_{\bullet}^{\otimes n}: n\in \mathbb{N}\right)$ (see
\cite[Theorem 4.3.1]{zbMATH07226371}). Let \(\psi^{0}\) be a
stochastic flow of measurable mappings in $M$ associated to the
consistent sequence of coalescing Feller transition functions
$(\mathsf{P}^{(n)}_{\bullet}:n\in \mathbb{N}).$

\begin{theorem}
  \label{thm:WBM_existence} There exists a strong measurable
  continuous modification of $\psi^0$.
\end{theorem}

\begin{proof}

  According to Theorem
  \ref{thm:metric_graph_instantaneous_coalescence} it is enough to
  verify (TF~\ref{item:tf_icp}), i.e. for any simple compact $K$,
  there are \(\beta,p>0\) such that that \eqref{eq:4_17}
  holds. Without loss of generality we will suppose that \(K\) is a
  neighborhood a vertex \(v\in{V}\) (in the case \(K\) doesn't contain
  any vertices, it suffices to add a vertex \(v\in K\) of degree 2 so
  that \(K\) is a compact neighborhood of \(v\), and to set the
  transmission parameters to $p_1=p_2=\frac{1}{2}$). Let $\tilde{M}$
  be a star graph that contains a simple compact isometric to
  $K$. Then $\tilde{M}$ is a metric graph with only one vertex $v$ and
  \(d=d(v)\) edges. We denote $v$ by $0$ and let the adjacent edges of
  $0$ be $E_1,\ldots,E_{d}$. Then for each $j\in \{1,\ldots,d\}$ there
  is a bijection $e_j:(0,\infty)\to E_j$ such that $e_j(0+)=0$. When
  $x=e_j(r)$, we set $|x|=r$ and define the distance on
  $\tilde{M}$ by:
  $$
  \rho(e_i(r),e_j(s))=
  \begin{cases}
    |r-s|, & i=j \\
    r+s, & i\ne j
  \end{cases}
  $$

  Denote by $p_1,\ldots,p_{d}\in [0,1]$ the transmissions parameters
  associated to $v$ and assign them to the edges of $\tilde{M}$. Let
  \(\mathbb{P}_{x,y}(=\mathbb{P}_{x}\otimes{}\mathbb{P}_{y})\) be the
  distribution of \((X,Y)\) where \(X\) and \(Y\) are two independent
  WBM's on $\tilde{M}$  started at \(x\) and \(y\), respectively. Set
  \(T_\Delta=\inf\{t:\,X_{t}=Y_{t}\}\). We note that if
  $\rho(x,y)\leq \epsilon$, then
  $\mathbb{P}^{(2)}_{x,y}[\sigma\wedge \tau >\beta \epsilon^2]\leq
  \mathbb{P}_{x,y}[T_\Delta>\beta \rho(x,y)^2]$. Hence, it is enough to
  show that for some $\beta>0,$
  \begin{equation}
    \label{eq:WBM_meeting_time}
    \inf_{(x,y)\in \tilde{M}^2}\mathbb{P}_{x,y}[T_{\Delta}\leq{\beta\rho (x,y)^{2}}]>0.
  \end{equation}
  The proof is separated into several lemmas.  We first consider the
  case when \(y=0\)\,:

  \begin{lemma}
    \label{lem:1} For any $\beta>0$, there is \(p=p(\beta)>0\) such that
    for all \(x\in{\tilde{M}}\),
    \begin{equation}
      \mathbb{P}_{x,0}[T_{\Delta}\leq{\beta |x|^{2}}]\geq{p}.
    \end{equation}
  \end{lemma}
  
  \begin{proof}
    The WBM on \(\tilde{M}\) is scaling invariant\,: if \(X\) is a WBM
    started at \(x\), then for all \(\lambda>0\), the process
    \(X^{\lambda}\) defined by
    \(X^{\lambda}(t)=\lambda^{-1}X(\lambda^{2}t)\) is a WBM started at
    \(\lambda{}x\).  Using this scaling property with
    \(\lambda=|x|^{-1}\), we obtain that
    \begin{equation*}
      \mathbb{P}_{x,0}[T_{\Delta}\leq{\beta |x|^{2}}]=\mathbb{P}_{u,0}[T_{\Delta}\leq{\beta}]
    \end{equation*}
    where \(u=x/|x|\). In other words, to prove the lemma it suffices to
    prove that for all \(j\), we have that
    \begin{equation}
      \mathbb{P}_{e_{j}(1),0}[T_{\Delta}\leq{\beta}]>0.
    \end{equation}

    We therefore fix \(j\) and let \(X,Y\) be distributed as
    \(\mathbb{P}_{e_{j}(1),0}\).  Let \(i\in\{1,\ldots,d\}\) be such
    that \(p_i>0\). With positive probability
    $(X(\frac{\beta}{2}),Y(\frac{\beta}{2}))\in E^2_i$. It remains to
    note that with positive probability during the time interval
    $[0,\frac{\beta}{2}]$ two independent Brownian motions (BMs) meet
    before leaving $\mathbb{R}_+$.
  \end{proof}

  We now consider the case where \(x,y\) belong to two different
  edges\,:
  \begin{lemma}
    \label{lem:2}
    There exists \(p>0\) such for all \((x,y)\in{E_{i}}\times{E_{j}}\)
    with \(i\neq{}j\),
    \begin{equation}
      \mathbb{P}_{x,y}[T_{\Delta}\leq{2\rho(x,y)^{2})}]\geq p.
    \end{equation}
  \end{lemma}

  \begin{proof}
    Let \((x,y)\in{E_{i}}\times{E_{j}}\) with \(i\neq{j}\). Note that
    \(\rho(x,y)=|x|+|y|\). Denote $\tau_Y=\inf\{t\geq
    0:\,Y(t)=0\}$. Then, using the strong Markov property at time
    \(\tau_{Y}\),
    \begin{align*}
      \mathbb{P}_{x,y}\left[T_{\Delta}\leq{}2\rho(x,y)^{2}\right]
      &= \mathbb{P}_{x,y}\left[T_{\Delta}\leq{}2(|x|+|y|)^{2}\right]\\
      &\geq
        \mathbb{P}_{x,y}\left[\tau_{Y}\leq{|y|^{2}},\,|X(\tau_{Y})|\leq{|x|+|y|}
        \text{ and }T_{\Delta}-\tau_{Y}\leq{}(|x|+|y|)^{2}\right]\\
      &\geq
        \mathbb{E}_{x,y}\left[1_{\{\tau_{Y}\leq{|y|^{2}}\}\cap\{|X(\tau_{Y})|\leq{|x|+|y|}\}}
        \mathbb{P}_{X(\tau_{Y}),0}\left[T_{\Delta}\leq{}(|x|+|y|)^{2}\right]\right]
    \end{align*}
    On the event \(\{|X(\tau_{Y})|\leq{|x|+|y|}\}\), we have that
    \begin{align*}
      \mathbb{P}_{X(\tau_{Y}),0}\left[T_{\Delta}\leq{}(|x|+|y|)^{2}\right]
      \geq{p(\beta)}\geq{}p(1)>0,    
    \end{align*}
    where $\beta=\frac{(|x|+|y|)^{2}}{|X(\tau_{Y})|^{2}}|\geq{1}$ and
    $p(1)$ is defined in Lemma \ref{lem:1}.  We therefore have that
    \begin{equation*}
      \mathbb{P}_{x,y}\left[T_{\Delta}\leq{}2\rho(x,y)^{2}\right]
      \geq{}p(1)\times{}\mathbb{P}_{x,y}\left[{\{\tau_{Y}\leq{|y|^{2}}\}\cap\{|X(\tau_{Y})|\leq{|x|+|y|}\}}\right].
    \end{equation*}
    Since \(|X|\) is a BM reflected at \(0\), we have that under
    \(\mathbb{P}_{x}\), \(|X|\) is distributed as \(||x|+B|\) where
    \(B\) is a Brownian motion started at \(0\). Setting
    \(\varphi(t)=\mathbb{P}\left[|B_{t}|\leq{1}\right]\) for all
    \(t>0\), we thus have that
    \begin{align*}
      \mathbb{P}_{x}\left[|X_{t}|\leq{|x|+|y|}\right]
      &= \mathbb{P}\left[||x|+B_{t}|\leq{|x|+|y|}\right]\\
      &\geq \mathbb{P}\left[|B_{t}|\leq{|y|}\right]
        = \varphi\left(\frac{t}{|y|^{2}}\right).
    \end{align*}
    We therefore have that, using again the scaling property for \(Y\),
    \begin{align*}
      \mathbb{P}_{x,y}\left[T_{\Delta}\leq{}2\rho(x,y)^{2}\right]
      &\geq{}p(1)\times{}\mathbb{E}_{y}\left[1_{\{\tau_{Y}\leq{|y|^{2}}\}}\times\varphi\left(\frac{\tau_{Y}}{|y|^{2}}\right)\right]\\
      &\geq{}p(1)\times{}\mathbb{E}_{e_{j}(1)}\left[1_{\{\tau_{Y}\leq{1}\}}\times\varphi\left({\tau_{Y}}\right)\right]>0.
    \end{align*}
    which proves the lemma.
  \end{proof}

  We finally consider the case where \(x,y\) both belong to the same
  edge\,: Denote by \(q\) the probability that two independent BMs
  started at distance one meet before time $1$. Then \(q\) is a positive
  probability.  Let \(\varepsilon>0\) be such that
  \begin{equation}
    \label{eq:1}
    \mathbb{P}[|B_{1}|\geq\varepsilon^{-1}]=\frac{q}{2}.
  \end{equation}
  \begin{lemma}
    \label{lem:2}
    There exists \(p>0\) such for all \((x,y)\in{E_{i}}\times{E_{i}}\),
    \begin{equation}
      \mathbb{P}_{x,y}[T_{\Delta}\leq{2\rho(x,y)^{2})}]\geq p.
    \end{equation}
  \end{lemma}
  \begin{proof}
    Let \((x,y)\in{E_{i}}\times{E_{i}}\) with \(0<|y|<|x|\). Note that
    \(\rho(x,y)=|x|-|y|\). Denote $\tau_Y=\inf\{t\geq 0:\,Y(t)=0\}$.

    The scaling property implies that, setting \(r=\frac{|x|}{|y|}\),
    \begin{align*}
      \mathbb{P}_{x,y}\left[T_{\Delta}\leq{}\rho(x,y)^{2}\right]
      &=\mathbb{P}^{i}_{r,1}\left[T_{\Delta}\leq{}(r-1)^{2}\right]
    \end{align*}
    where \(\mathbb{P}^{i}_{r,1}=\mathbb{P}_{e_{i}(r),e_{i}(1)}\).

    Further,
    \begin{align*}
      \mathbb{P}^{i}_{r,1}\left[T_{\Delta}\leq{}(r-1)^{2}\right]
      &\geq{}
        \mathbb{P}^{i}_{r,1}\left[T_{\Delta}\leq{}(r-1)^{2}\leq\tau_{Y}\right]\\
      &={} \mathbb{Q}_{r,1}\left[T_{\Delta}\leq{}(r-1)^{2}\leq\tau_{Y}\right]
    \end{align*}
    where under \(\mathbb{Q}_{r,1}\), \(X\) and \(Y\) are two
    independent Brownian motions respectively started at \(r\) and
    \(1\). We thus have that
    \begin{align*}
      \mathbb{P}^{i}_{r,1}\left[T_{\Delta}\leq{}(r-1)^{2}\right]
      &\geq{} \mathbb{Q}_{r,1}\left[T_{\Delta}\leq{}(r-1)^{2}\right]
        - \mathbb{Q}_{r,1}\left[\tau_{Y}\leq{}(r-1)^{2}\right]
    \end{align*}
    We have that (using again the scaling property)
    \begin{align*}
      \mathbb{Q}_{r,1}\left[T_{\Delta}\leq{}(r-1)^{2}\right]
      &= \mathbb{Q}_{r-1,0}\left[T_{\Delta}\leq{}(r-1)^{2}\right]\\
      &= \mathbb{Q}_{1,0}\left[T_{\Delta}\leq{}1\right]= q >0.
    \end{align*}

    Using the reflection principle, we have that
    \begin{align*}
      \mathbb{Q}_{r,1}\left[\tau_{Y}\leq{}(r-1)^{2}\right]
      &= 2\mathbb{P}\left[B_{(r-1)^{2}}\leq{}-1\right] \\
      &= \mathbb{P}\left[|B_{1}|\geq{(r-1)^{-1}}\right].
    \end{align*}
    where under $\mathbb{P},$ \(B\) is a BM started at \(0\). For all
    \(r\in(1,1+\varepsilon]\), \eqref{eq:1} implies that
    \(\mathbb{P}\left[|B_{1}|\geq (r-1)^{-1}\right]\leq\frac{q}{2}\).
    We therefore have that for all \(r\in(1,1+\varepsilon]\),
    \begin{align*}
      \mathbb{P}^{i}_{r,1}\left[T_{\Delta}\leq{}(r-1)^{2}\right]
      &\geq q-\frac{q}{2}=\frac{q}{2}.
    \end{align*}

    Let us suppose now that \(r>1+\varepsilon\). In this case, using the
    strong Markov property at time \(\tau_{Y}\),
    \begin{align*}
      \mathbb{P}^{i}_{r,1}\left[T_{\Delta}\leq{}2(r-1)^{2}\right]
      &\geq
        \mathbb{P}^{i}_{r,1}\left[\tau_{Y}\leq{\varepsilon^{2}},\,|X_{\tau_{Y}}|\leq{1+r},\,T_{\Delta}-\tau_{Y}\leq{}(r-1)^{2}\right]
      \\
      &\geq
        \mathbb{E}^{i}_{r,1}\left[1_{\{\tau_{Y}\leq{\varepsilon^{2}}\}\cap\{|X_{\tau_{Y}}|\leq{1+r}\}}\mathbb{P}_{X_{\tau_{Y}},0}[T_{\Delta}\leq{}(r-1)^{2}]\right]
    \end{align*}
    Set
    \(\beta_{\varepsilon}=\inf_{r>1+\varepsilon}\left\{\frac{(r-1)^{2}}{(1+r)^{2}}\right\}>0\). On
    the event \(\{|X(\tau_{Y})|\leq{1+r}\}\), we have
    \begin{align*}
      \mathbb{P}_{X(\tau_{Y}),0}\left[T_{\Delta}\leq{}(r-1)^{2}\right]
      &\geq{}\mathbb{P}_{u,0}\left[T_{\Delta}\leq{}\beta_{\varepsilon}(1+r)^{2}\right]\\
      &\geq{}p(\beta_{\varepsilon})>0,    
    \end{align*}
    where $|u|=1+r.$ We therefore have that
    \begin{equation*}
      \mathbb{P}^{i}_{r,1}\left[T_{\Delta}\leq{}2(r-1)^{2}\right]
      \geq{}p(\beta_{\varepsilon})\times{}\mathbb{P}^{i}_{r,1}\left[{\{\tau_{Y}\leq{\varepsilon^{2}}\}\cap\{|X(\tau_{Y})|\leq{1+r}\}}\right].
    \end{equation*}
    Since \(|X|\) is a BM reflected at \(0\), we have that under
    \(\mathbb{P}^{i}_{r,1}\), \(|X|\) is distributed as \(|r+B|\) with
    \(B\) a BM started at \(0\). We thus have that
    \begin{align*}
      \mathbb{P}^{i}_{r,1}\left[|X_{t}|\leq{1+r}\right]
      &= \mathbb{P}\left[|r+B_{t}|\leq{r+1}\right]\\
      &\geq \mathbb{P}\left[|B_{t}|\leq{1}\right] = \varphi(t)
    \end{align*}

    Finally,
    \begin{displaymath}
      \mathbb{P}^{i}_{r,1}\left[T_{\Delta}\leq{}2\rho(x,y)^{2}\right]
      \geq{}
      p(\beta_{\varepsilon})\times{}
      \mathbb{E}^i_{r,1}\left[1_{\{\tau_{Y}\leq{\varepsilon^{2}}\}}\times\varphi(\tau_{Y})\right]
      >0
    \end{displaymath}
    which proves the lemma.
  \end{proof}

  This proves \eqref{eq:WBM_meeting_time} and the Theorem.
\end{proof}

\subsection{Coalescing Tanaka flow} 
\label{subsec:Coalescing Tanaka flow}
In this section we consider a stochastic flow of measurable mappings
in $\mathbb{R}$ that consists of solutions to Tanaka's SDE
\begin{equation}
  \mathrm{d}X(t)=\mbox{sign}(X(t))\,\mathrm{d}W(t),
\end{equation}
where $W$ is a Brownian motion and
\(\mbox{sign}(x)=
\begin{cases}
  1 &\mbox{if } x\geq 0 \\
  -1 &\mbox{if } x<0
\end{cases}
\). It is well-known that Tanaka's SDE does not possess a strong
solution. When two solutions meet they may not coalesce. Let us assume
coalescence in the definition of the $n$-point motions. More
precisely, the \(n\)-point motion \((X_{1},\ldots{},X_{n})\) starting
from $x\in \mathbb{R}^n$ solves the following problem
\begin{displaymath}
  \begin{cases}
    X_i(t)=x_i+\int^t_0 \mbox{sign}(X_i(s))\,\mathrm{d}W(s), & 1\leq i\leq n, \\
    X_i(s)=X_j(s)\:\Rightarrow\: X_i(t)=X_j(t), & 1\leq i<j\leq n,\
                                                  s\leq{t}
  \end{cases}
\end{displaymath}
The problem has a unique weak solution, which is a Feller process. We
denote its transition function by \(\mathsf{P}^{(n)}_{\bullet{}}\).  The
sequence $(\mathsf{P}^{(n)}_{\bullet{}}:n\in \mathbb{N})$ is a
sequence of coalescing Feller transition functions that satisfies
(TF~\ref{item:consistency}), (TF~\ref{item:coalescing_property}) and
(TF~\ref{item:continuity_of_trajectories_TF3}). Let $\psi^0$ be a
stochastic flow of measurable mappings in $\mathbb{R}$ associated to
$(\mathsf{P}^{(n)}_{\bullet{}}:n\in \mathbb{N})$. Let
$\mathcal{D}=\{(s_n,x_n):n\in\mathbb{N}\}$ be a countable dense set in
$\mathbb{R}\times\mathbb{R}$, and let \(\varPhi\) be the random
sequence constructed in Section~\ref{sec:rand-sequ-varphi}. Recall
that for each \(n\), $\varPhi_n$ is a continuous modification of
$\psi^0_{s_n,\cdot}(x_n)$ with $\varPhi_n(s_n)=x_n$.  In Lemma
\ref{lem:Tanaka_TF4} we verify that (TF~\ref{item:14}) is satisfied
and thus we can construct a measurable continuous modification
$\theta$ of $\psi^0$ as in
Section~\ref{sec:measurable_modification_out_of_skeleton}.  Let
$W_n(t)=\int^t_{s_n}\mbox{sign}(\varPhi_n(s))\,\mathrm{d}\varPhi_n(s),$
$t\geq s_n.$ Note that a.s.
\begin{equation}
  \label{eq:Tanaka_same_B}
  (n,m)\in \mathbb{N}^2,\, s\geq s_n\vee s_m
  \: \Rightarrow\:
  W_n(t)-W_n(s)=W_m(t)-W_m(s),\,
  t\geq s.
\end{equation}
Denote $\sigma_n=\inf\{t\geq s_n:\varPhi_n(t)=0\}.$

\begin{lemma}
  \label{lem:Tanaka_TF4} Let $L\subset \mathbb{R}\times \mathbb{R}$ be a
  compact set. Then a.s. the set
  \begin{displaymath}
    \{\varPhi_n[s,\infty): n\in I^s, (s,\varPhi_n(s))\in L)\}
  \end{displaymath}
  is relatively compact in $X$.
\end{lemma}

\begin{proof} For a continuous real-valued function $f$ denote its
  modulus of continuity on $[a,b]$ by $m_{a,b}(\delta;f):$
  $$
  m_{a,b}(\delta;f)=\sup\{|f(t)-f(s)|: (s,t)\in [a,b]^2, |t-s|\leq
  \delta\}.
  $$
  By Tanaka's formula and Skorokhod's reflection lemma (\cite[Ch. VI,
  Th. (1.2), L. (2.1)]{zbMATH02150787}),
  \begin{equation}
    \label{eq:Tanaka_solution}
    |\varPhi_n(t)|=\begin{cases}
      |x_n|+W_n(t), & t\in[s_n,\sigma_n); \\
      W_n(t)-\inf_{s\in [\sigma_n,t]}W_n(s) & t\in[\sigma_{n},\infty).
    \end{cases}
  \end{equation}
  Hence, a.s. for any $a,b,$ $s_n\leq a<b,$ we have that
  $m_{a,b}(\delta;|\varPhi_n|)\leq 2m_{a,b}(\delta;W_n).$ Since
  $\varPhi_n$ is continuous, it follows that
  $m_{a,b}(\delta,\varPhi_n)\leq 4m_{a,b}(\delta,W_n).$

  Let $L\subset \mathbb{R}\times \mathbb{R}$ be a compact set. Let $n_0$
  and $C$ be such that $s_{n_0}<s<C$ for all $(s,x)\in L.$ Given
  $\varepsilon>0$ there exists $\delta>0$ such that
  $m_{s_{n_0},C}(\delta,W_{n_0})\leq \varepsilon.$ Assume that
  $(s,\varPhi_n(s))\in L.$ Then $s\geq s_n\vee s_{n_0},$ hence
  $W_n(t)-W_n(s)=W_{n_0}(t)-W_{n_0}(s),$ $t\geq s.$ In particular,
  $m_{s,C}(\delta,W_n)=m_{s,C}(\delta,W_{n_0}).$ It follows that if
  $(t_1,t_2)\in [s,C]^2$ and $|t_1-t_2|\leq \delta,$ then
  $$
  |\varPhi_n(t_1)-\varPhi_n(t_2)|\leq 4m_{s,C}(\delta,W_n)\leq
  4m_{s_{n_0},C}(\delta,W_{n_0})\leq 4\varepsilon.
  $$
  The result follows from Theorem \ref{lem:ascoli-arzela}.
\end{proof}

Lemma~\ref{lem:Tanaka_TF4} ensures that a.s. \(\varPhi\) is a
skeleton. 

\begin{lemma}
  \label{lem:Tanaka_shell}
  $F:=\mathbb{R}\times\{0\}$ is a closed shell of \(B(\varPhi)\), the
  set of bifurcation points of $\varPhi.$
\end{lemma}

\begin{proof} Let \((s,x)\in\mathbb{R}\times(0,\infty)\). There exist
  $n\in I^s$ and $t>s$ such that $\varPhi_n([s,t])\subset (0,x).$ It
  is easy to see from \eqref{eq:Tanaka_solution} and
  \eqref{eq:Tanaka_same_B} that $\mathcal{K}^{s,t}_x$ contains only
  the single function $f$ defined by
  $$
  f(r)=x+\varPhi_n(r)-\varPhi_n(s), \ r\in[s,t].
  $$
  This proves that $\tau^s_x> t>s$ and that
  $(s,x)\not\in{}B(\varPhi)$. Similarly,
  \(\mathbb{R}\times(-\infty,0)\subset{B(\varPhi)^{c}}\). Thus
  \(B(\varPhi)\subset{F}\).
\end{proof}

\begin{theorem}
  \label{thm:Tanaka_strong_flow}
  There exists a strong measurable continuous modification of $\psi^0$.
\end{theorem}

\begin{proof}
  Since \(\varPhi\) is a skeleton a.s., we can construct \(\theta\), the
  measurable continuous modification of $\psi^0$, as in
  Section~\ref{sec:measurable_modification_out_of_skeleton}.

  To prove this theorem we apply Theorem~\ref{thm:sfp_stochastic}. We
  take for closed shell of \(B(\varPhi)\) the set
  \(F=\mathbb{R}\times\{0\}\) and as in
  Section~\ref{sec:sf-meas-modif} define out of \(\theta\) the
  stopping times $\sigma^s_x(k)$ and the random variables $z^s_x(k)$,
  $k^s_x$. Note that \(z^{s}_{x}(0)=x\) and \(z^{s}_{x}(k)=0\) if
  \(k\geq{1}\). To apply Theorem~\ref{thm:sfp_stochastic}, we have to
  verify that a.s. for all \((s,x)\in\mathbb{R}\times\mathbb{R}\),
  $k^s_x<\infty$, and if $k^s_x=0$, then for every $t>s$ there is
  $n\in I^t$ such that
  \(\theta_{s,\cdot}(x)[t,\infty)=\varPhi_{n}[t,\infty)\).
  
  We note that with probability 1 for every
  $(s,x)\in \mathbb{R}\times \mathbb{R}$ and every $n\in I^s,$
  \begin{displaymath}
    s\leq a<b, \inf_{t\in [a,b]}\left|\theta_{s,t}(x)\right|>0
    \quad \Longrightarrow \quad
    \left|\theta_{s,t}(x)\right|
    = \left|\theta_{s,a}(x)\right|+W_n(t)-W_n(a),
    \ t\in [a,b].
  \end{displaymath}

  Suppose that $k^s_x\geq 3$ for some
  $(s,x)\in \mathbb{R}\times \mathbb{R}$ and let $n\in I^s$. Then
  $s=\sigma^s_x(0)<\sigma^s_x(1)<\sigma^s_x(2)<\sigma^s_x(3).$ It
  follows that $|\theta_{\sigma^s_x(j),t}(z^s_x(j))|>0$ for
  $t\in (\sigma^s_x(j),\sigma^s_x(j+1)),$ $j\in\{0,1,2\},$ and
  $z^s_x(1)=z^s_x(2)=0.$ Hence, $|x|+W_n(t)-W_n(s)>0$ for
  $t\in (s,\sigma^s_x(1)),$ $W_n(t)-W_n(\sigma^s_x(1))>0$ for
  $t\in (\sigma^s_x(1),\sigma^s_x(2)),$ $W_n(t)-W_n(\sigma^s_x(2))>0$
  for $t\in (\sigma^s_x(2),\sigma^s_x(3)),$ and
  $W_n(\sigma^s_x(1))=W_n(\sigma^s_x(2))=W_n(s)-|x|.$ It follows that
  $W_n$ has two local minima at the level $W_n(s)-|x|,$ which is a.s.
  impossible (see \cite{zbMATH01625446}). As a consequence, we have
  that a.s., \(k^{s}_{x}\leq{2}\) for all
  \((s,x)\in\mathbb{R}\times\mathbb{R}\).
  
  Suppose that $k^s_x=0$ for some
  $(s,x)\in \mathbb{R}\times \mathbb{R}$. Then $\sigma^s_x(1)=s$ and
  $x=0$. If $\varPhi_n(s)=0$ for some $n\in I^s$, then for all
  \(t>s\), $\theta_{s,t}(0)=\varPhi_n(t)$. Otherwise, there exists a
  sequence $(n_j:j\in \mathbb{N})$ in $I^s$, such that
  $\lim_{j\to\infty}\varPhi_{n_j}[s,\infty)=\theta_{s,\cdot}(0)$ in
  $X$. Without loss of generality we can assume that
  $\varPhi_{n_j}(s)>0$ for all $j\in \mathbb{N}$, and
  $\varPhi_{n_j}(s)\downarrow 0$, $j\to\infty$. Assume that
  $\inf_{r\in [s,t]}\varPhi_{n_j}(r)>0$ for all $j.$ Then
  $\varPhi_{n_j}(r)=\varPhi_{n_j}(s)+W_{n_1}(r)-W_{n_1}(s),$
  $r\in [s,t],$ and $\theta_{s,r}(0)=W_{n_1}(r)-W_{n_1}(s),$
  $r\in [s,t]$. It follows that
  \begin{displaymath}
    \inf_{r\in[s,t]}W_{n_1}(r)+\varPhi_{n_j}(s)-W_{n_1}(s)>0,
  \end{displaymath}
  and $W_{n_1}(s)=\inf_{r\in [s,t]}W_{n_1}(r)$. Since
  $s=\sigma^s_x(1)$, there exist distinct $u, v$ in $(s,t)$ such that
  $\theta_{s,u}(0)=\theta_{s,v}(0)=0$. Then $u$ and $v$ are both local
  extrema of $W_{n_1}$ with $W_{n_1}(u)=W_{n_1}(v)$. This is
  a.s. impossible (see \cite{zbMATH01625446}). Hence,
  $\varPhi_{n_{j_0}}(r)=0$ for some $j_0\geq 1$ and $r\in (s,t].$ It
  follows that $\varPhi_{n_j}(r)=0=\varPhi_{n_{j_0}}(r)$ for all
  $j\geq j_0,$ and
  $\theta_{s,\cdot}(0)[r,\infty)=\varPhi_{n_{j_0}}[r,\infty).$ This
  verifies the condition of Theorem \ref{thm:sfp_stochastic} and
  proves that $\psi$ is a strong measurable continuous modification of
  $\psi^0$.
\end{proof}

\subsection{Burdzy-Kaspi flows}
\label{subsec:BK flow}
In this section we consider a stochastic flow of measurable mappings
in $\mathbb{R}$ that consists of solutions to the Harrison-Shepp SDE
for the skew Brownian motion
\begin{equation}
  \mathrm{d}X(t)=\mathrm{d}W(t)+\beta \mathrm{d}L(t),
\end{equation}
where $W=(W(t):t\in\mathbb{R})$ is a Brownian motion on $\mathbb{R}$,
$L$ is the symmetric local time of $X$ at zero and $\beta\in
[-1,1]$. It is well-known that for all
$(s,x)\in \mathbb{R}\times \mathbb{R}$ the equation
\begin{equation}
  \label{eq:SBM_SDE}
  \begin{cases}
    X(t)=x+W(t)-W(s)+\beta L(t),  & t\geq s \\
    L(t)=\lim_{\varepsilon\to 0}\frac{1}{2\varepsilon}\int^t_s
    1_{(-\varepsilon,\varepsilon)}(X(r))\mathrm{d}r,  & t\geq s
  \end{cases}
\end{equation}
has a unique strong solution (see \cite{zbMATH03723641}). Define 
$$
\mathsf{P}^{(n)}_t(x,B)=\mathbb{P}[(X_1(t),\ldots,X_n(t))\in B],
$$
where $x\in \mathbb{R}^n$, $B\in \mathcal{B}(\mathbb{R}^n)$,
$t\geq 0$, and $X_i$ is the solution of \eqref{eq:SBM_SDE} with
initial condition $X_i(0)=x_i$, $i\in \{1,\ldots,n\}$. The sequence
$(\mathsf{P}^{(n)}_{\bullet{}}:n\in \mathbb{N})$ is then a sequence of
coalescing Feller transition functions that satisfies
(TF~\ref{item:consistency}), (TF~\ref{item:coalescing_property}) and
(TF~\ref{item:continuity_of_trajectories_TF3}). Let $\psi^0$ be a
stochastic flow of measurable mappings in $\mathbb{R}$ associated to
$(\mathsf{P}^{(n)}_{\bullet{}}:n\in \mathbb{N})$. Let
$\mathcal{D}=\{(s_n,x_n):n\in\mathbb{N}\}$ be a countable dense set in
$\mathbb{R}\times\mathbb{R}$, and let $\varPhi$ be the random sequence
constructed in Section~\ref{sec:rand-sequ-varphi}. Recall that for
each \(n\), $\varPhi_n$ is a continuous modification of
$\psi^0_{s_n,\cdot}(x_n)$ with $\varPhi_n(s_n)=x_n$. By
\cite[Prop. 2.1]{zbMATH02148685}, the condition (TF~\ref{item:14}) is
satisfied and $\varPhi$ is a.s. a skeleton. Denote
$L_n(t)=\lim_{\varepsilon\to 0}\frac{1}{2\varepsilon}\int^t_s
1_{(-\varepsilon,\varepsilon)}(\varPhi_n(r))dr,$
$W_n(t)=\varPhi_n(t)-x_n-\beta L_n(t),$ $t\geq s_n.$ Note that a.s.
\begin{equation}
  \label{eq:HS_same_B}
  (n,m)\in \mathbb{N}^2,\, s\geq s_n\vee s_m
  \quad  \Longrightarrow \quad
  W_n(t)-W_n(s)=W_m(t)-W_m(s),\, t\geq s.
\end{equation}

\begin{theorem}
  \label{thm:BK_strong_flow}
  There exists a strong measurable continuous modification of $\psi^0$. 
\end{theorem}

\begin{proof}
  
  Similarly to the proof of Lemma~\ref{lem:Tanaka_shell} it can be
  checked that $F=\mathbb{R}\times \{0\}$ is a closed shell of
  \(B(\varPhi)\), the set of bifurcation points of $\varPhi$. We then
  construct \(\theta\), the measurable continuous modification of
  $\psi^0$, as in
  Section~\ref{sec:measurable_modification_out_of_skeleton}.

  As in Section~\ref{sec:sf-meas-modif} we define out of \(\theta\)
  the stopping times $\sigma^s_x(k)$ and the random variables
  $z^s_x(k)$, $k^s_x$. Note \(z^{s}_{x}(0)=x\) and \(z^{s}_{x}(k)=0\)
  if \(k\geq 1\). To apply Theorem~\ref{thm:sfp_stochastic}, we have
  to verify that a.s. for all \((s,x)\in\mathbb{R}\times\mathbb{R}\),
  $k^s_x<\infty$, and if $k^s_x=0$, then for every $t>s$ there is
  $n\in I^t$ such that
  \(\theta_{s,\cdot}(x)[t,\infty)=\varPhi_{n}[t,\infty)\).
  
  The case $\beta=0$ is trivial, and in this case
  $\psi_{s,t}(x)=\theta_{s,t}(x)=x+W_m(t)-W_m(s)$ for all $s\leq t,$
  $x\in \mathbb{R}$ and $m\in I^s$. We consider the case $\beta\ne 0$.
  
  Suppose that $k^s_x\geq 3$ for some
  $(s,x)\in \mathbb{R}\times \mathbb{R}$ and $n\in I^s$. Then
  $s=\sigma^s_x(0)<\sigma^s_x(1)<\sigma^s_x(2)<\sigma^s_x(3)$. It
  follows that $\theta_{\sigma^s_x(j),t}(z^s_x(j))\ne 0$ for
  $t\in (\sigma^s_x(j),\sigma^s_x(j+1))$, $j\in\{0,1,2\},$ and
  $z^s_x(1)=z^s_x(2)=0$. Hence, $x+W_{n}(t)-W_{n}(s)\ne 0$ for
  $t\in (s,\sigma^s_x(1))$, $W_{n}(t)-W_{n}(\sigma^s_x(1))\ne 0$ for
  $t\in (\sigma^s_x(1),\sigma^s_x(2))$,
  $W_{n}(t)-W_{n}(\sigma^s_x(2))\ne 0$ for
  $t\in (\sigma^s_x(2),\sigma^s_x(3)),$ and
  $W_{n}(\sigma^s_x(1))=W_{n}(\sigma^s_x(2))=W_{n}(s)-x$. The Brownian
  motion has no points of increase or decrease
  \cite{zbMATH03179897}. Hence, $W_{n}$ has two local extrema at the
  level $W_{n}(s)-x$, which is a.s. impossible (see
  \cite{zbMATH01625446}). As a consequence, we have that a.s.,
  \(k^{s}_{x}\leq{2}\) for all \((s,x)\in\mathbb{R}\times\mathbb{R}\).
  
  Suppose that $k^s_x=0$ for some
  $(s,x)\in\mathbb{R}\times\mathbb{R}$. Then $\sigma^s_x(1)=s$ and
  $x=0$. Let us fix \(m\in I^{s}\). If $\varPhi_n(s)=0$ for some
  $n\in I^s$, then for all \(t>s\),
  $\theta_{s,t}(0)=\varPhi_n(t)$. Otherwise, either
  \begin{displaymath}
    \theta_{s,t}(0)=\inf\{\varPhi_{n}(t): n\in I^s, s_n<x,
    \varPhi_n(s)>x\}, \ \forall t\geq s,
  \end{displaymath}
  or 
  \begin{displaymath}
    \theta_{s,t}(0)=\sup\{\varPhi_{n}(t): n\in I^s, s_n<x,
    \varPhi_n(s)<x\}, \ \forall t\geq s.
  \end{displaymath}

  In either case, $\theta_{s,t}(0)=W_m(t)-W_m(s)+\beta L_{s,t}$, where
  $L_{s,t}= \lim_{\varepsilon\to{}0} \frac{1}{2\varepsilon}
  \int^t_s1_{(-\varepsilon.\varepsilon)}(\theta_{s,r}(0)) \mathrm{d}r$
  (see \cite[Prop. 1.1]{zbMATH02148685}). Since $\sigma^s_x(1)=s$,
  there exists a sequence $u_1>u_2>\ldots$ in $(s,t),$ such that
  $\theta_{s,u_n}(0)=0$ for all $n\in \mathbb{N}$. If $L_{s,t}=0,$
  then $L_m(t)-L_m(s)=0$. Hence, the Brownian motion $W_m$ hits the
  level $a=W_m(s)$ infinitely often in the interval $(s,t),$ while its
  local time at the level $a$ does not increase between times $s$ and
  $t,$ which is impossible by \cite{zbMATH03192471}.  It follows that
  $L_{s,t}>0.$ By \cite[Lemma 2.7]{zbMATH02148685} there exists
  $n\in I^t,$ such that
  $\theta_{s,\cdot}(x)[t,\infty)=\varPhi_n[t,\infty).$ This verifies
  the condition of Theorem \ref{thm:sfp_stochastic}, and proves that
  $\psi$ is a strong measurable continuous modification of $\psi^0$.
\end{proof}

\subsection{Coalescing Tanaka flow on a star graph}
\label{subsec:Tanaka_graph}

In this section, we let $M$ be a star graph, i.e. a metric graph with
only one vertex $0$ and \(d\geq 2\) edges $E_1,\ldots,E_{d}$, and we
consider a stochastic flow of measurable mappings in $M$, that
consists of solutions to Tanaka's SDE on $M$.  For each
$j\in \{1,\ldots,d\}$ there is a bijection $e_j:(0,\infty)\to E_j$
such that $e_j(0+)=0$. When $x=e_j(r)$, we set $|x|=r$ and define the
distance on $M$ by:
\begin{displaymath}
  \rho(e_i(r),e_j(s))=
  \begin{cases}
    |r-s|, & i=j \\
    r+s, & i\ne j
  \end{cases}
\end{displaymath}

Denote by $p_1,\ldots,p_{d}\in [0,1]$ the transmissions parameters
associated to the vertex $0$ and edges $E_1,\ldots,E_d,$ respectively.
For any function $f:M\to \mathbb{R}$ we introduce functions
$f_j=f\circ e_j:(0,\infty)\to \mathbb{R},$ $j\in \{1,\ldots,d\}$.
Denote by $C^2_0(M)$ the space of continuous functions $f:M\to \mathbb{R},$ such
that for all $j\in\{1,\ldots,d\},$ $f_j\in C^2((0,\infty)),$
$f_j,f'_j,f''_j$ are bounded, and all limits $f_j(0+),$ $ f'_j(0+),$
$f''_j(0+)$ exist. For $f\in C^2_0(M)$ we let $f'(x)=f'_j(|x|),$
$f''(x)=f''_j(|x|),$ if $x\in E_j$. Let
$D=\{f\in C^2_0(M): \sum^d_{j=1}p_j f'_j(0+)=0\}.$ To each vertex
$E_j$ we associate a sign $\varepsilon_j\in\{-1,1\},$ in such a way
that $\varepsilon_1=\ldots=\varepsilon_l=1,$
$\varepsilon_{l+1}=\ldots=\varepsilon_d=-1,$ $l\in \{1,\ldots, d-1\}$.
Let
\begin{displaymath}
  \varepsilon(x)=
  \begin{cases}
    1,  &\ x\in E_j, \ j\leq l \\
    -1, &\ x\in E_j, \ j>l \mbox{ or } x=0
  \end{cases}.
\end{displaymath}
There exists a stochastic flow $\psi^0$ of measurable mappings in $M$,
such that for all $(s,x)\in \mathbb{R}\times M$ and all $f\in D$,
\begin{displaymath}
  f(\psi^0_{s,t}(x)) =
  f(x) +
  \int_{s}^{t} \varepsilon(\psi^0_{s,r}(x))f'(\psi^0_{s,r}(x))\mathrm{d}W(r)
  +\frac{1}{2}\int^t_s f''(\psi^0_{s,r}(x))\mathrm{d}r, \ t\geq s,
\end{displaymath}
where $W$ is a Brownian motion on $\mathbb{R}$ \cite{zbMATH06049114}.
Let $\mathcal{D}=\{(s_n,x_n):n\in\mathbb{N}\}$ be a countable dense
set in $\mathbb{R}\times\mathbb{R}$. For each \(n\in\mathbb{N}\), we
let $\varPhi_n$ be a continuous modification of
$\psi^0_{s_n,\cdot}(x_n)$ with $\varPhi_n(s_n)=x_n$. The proof of
Lemma \ref{lem:Tanaka_TF4} shows that $\varPhi$ is a.s. a skeleton.

\begin{theorem}
  \label{thm:Tanaka_graph_strong_flow}
  There exists a strong measurable continuous modification of $\psi^0$. 
\end{theorem}

\begin{proof}
  Similarly to the proof of Lemma~\ref{lem:Tanaka_shell} it can be
  checked that $F=\mathbb{R}\times \{0\}$ is a closed shell of
  $B(\varPhi)$, the set of bifurcation points of $\varPhi$. Let
  $\theta$ be the measurable continuous modification of $\psi^0$
  defined from $\varPhi$ and the mapping $\Theta$. As in
  Section~\ref{sec:sf-meas-modif} we define out of $\theta$ the
  stopping times $\sigma^s_x(k)$, and the random variables $z^s_x(k),$
  $k^s_x.$ Consider the mapping $G:M\to \mathbb{R}$,
  $G(x)=\varepsilon(x)|x|$. To have notations consistent with
  \cite{zbMATH06049114}, we now suppose without loss of generality
  that $ \mathcal{D}=\mathbb{Q}\times M_{\mathbb{Q}}$, where
  $M_{\mathbb{Q}}$ is a countable dense set in $M$. Out of the random
  skeleton $\varPhi$, we construct a family of random mappings
  $\mathcal{Y}$ such that $\mathcal{Y}_n =G(\varPhi_n)$. The family
  $\mathcal{Y}$ is a random skeleton of a Burdzy-Kaspi flow (see
  \cite{zbMATH06049114}, in which $\psi^0$ is constructed with a
  Burdzy-Kaspi flow). Let us now define $Y,$ the Burdzy-Kaspi flow,
  out of the skeleton $\mathcal{Y}$ in the same way as $\psi^0$ is
  constructed out of $\varPhi$ (note that
  $G(\psi^0_{s,t}(x)) =Y_{s,t}(G(x))$). The proof of
  Theorem~\ref{thm:BK_strong_flow} implies that a.s. for all
  $(s,x)\in \mathbb{R}\times M,$ $k^s_x\leq 2$ (the random variables
  $k^s_x$ and $k^s_{\varepsilon(x)|x|}$ are equal, where 
  $k^s_{\varepsilon(x)|x|}$ is defined out of $Y$).

  Suppose that $k^s_x= 0$. Then $x= 0$ and
  $\sigma^s_x(1)=s$. The proof of  Theorem~\ref{thm:BK_strong_flow} implies that a.s. the local time  at zero $L_{s,r}$ of $Y_{s,\cdot}(0)$ is positive for all $r>s.$  Following the proof of \cite[Prop. 1.1(iii)]{zbMATH02148685}, for all  small
  enough positive $\varepsilon,$
  $\{\varPhi_n(t) : \ n\in I^s, \rho(\varPhi_n(s),0)< \varepsilon\}$
  is a finite set.  Since there is a sequence
  $(n_k: \ k\in \mathbb{N}),$ such that $\varPhi_{n_k}[s,\infty)$
  converges towards $\theta_{s,\cdot}(0)$ in $X$, the sequence
  $\varPhi_{n_k}[t,\infty)$ is stationary. This shows that there is
  $n\in I^t$ in such that $\theta_{s,u}(0) =\varPhi_n(u)$ for all
  $u\geq t$. Hence, conditions of Theorem~\ref{thm:sfp_stochastic} are
  verified and there exists a strong measurable continuous
  modification of $\psi^0.$
\end{proof}

The following Corollary extends the result of Theorem~\ref{thm:Tanaka_graph_strong_flow} to general metric graphs. For the definition of the Tanaka SDE on a metric graph we refer to \cite{zbMATH06257629}. 

\begin{corollary}
\label{cor:Tanaka_metric_graph}    There exists a strong measurable continuous modification of a stochastic flows of measurable mappings in a metric graph $M$, whose trajectories are solutions to the Tanaka SDE on $M.$
\end{corollary}

\begin{remark}
  \label{rem:interface_SDE} We expect that the proposed approach
  implies the existence of strong measurable continuous modifications
  of stochastic flows in metric graphs, that consist of solutions to interface SDE's in the sense of \cite{zbMATH06518035}.
  
\end{remark}

\section{Appendix. Measurable selection Lemma}
\label{sec:selection}

\begin{lemma}
  \label{lem:selection}
  Let $(X,d)$ be a complete separable metric space. Then, there exists a
  measurable function $\ell:X^\mathbb{N}\to X$ such that
  $$
  \overline{\cup^\infty_{n=1}\{x_n\}} \mbox{ compact } \Rightarrow
  \ell((x_n: n\in \mathbb{N})) \mbox{ is a limit point of } (x_n: n\in \mathbb{N}).
  $$
\end{lemma}

\begin{proof}
  Consider the compact metric space $K=[0,1]^\mathbb{N}$. From
  \cite[Lemma 1.1]{zbMATH02100692} it follows that there exists a
  measurable mapping $\tilde{\ell}:K^\mathbb{N}\to K$ such that for
  any sequence $(x_n: n\in \mathbb{N})$ the point
  $\tilde{\ell}((x_n: n\in \mathbb{N}))$ is a limit point of
  $(x_n: n\in \mathbb{N})$.

  There exists a homemorphism $f$ of $X$ onto a Borel subset of $K$
  \cite[Remark 2.2.8]{zbMATH01166155}. Denote by $g$ a Borel mapping
  of $K$ onto $X,$ such that $g(f(x))=x,$ $x\in X.$ Define
  $\ell:X^\mathbb{N}\to X$ by
  $\ell=g\circ \tilde{\ell}\circ f^{\otimes \mathbb{N}}.$ Consider a
  sequence $(x_n:n\in \mathbb{N})$ in $X$ such that
  $V=\overline{\cup^\infty_{n=1}\{x_n\}}$ is compact. Then $f(V)$ is a
  compact subset of $K,$ and
  $y=\tilde{\ell}\left((f(x_n):n\in \mathbb{N})\right)\in f(V).$
  Extracting subsequences we may assume that
  $y=\lim_{k\to\infty} f\left(x_{n_k}\right)$ and that the limit
  $\lim_{k\to\infty}x_{n_k}=:x_*\in V$ exists. It follows that
  $f(x_*)=y,$ and $\ell((x_n:n\in \mathbb{N}))=g(y)=x_*.$

  This proves Lemma \ref{lem:selection}.
\end{proof}

\noindent \textbf{Acknowledgments:} 
\begin{itemize}
\item This research has been conducted within the Fédération
  Parisienne de Modélisation Mathématique (FP2M)–CNRS FR 2036.

\item   This research has been conducted as part of the project Labex
  MME-DII (ANR11-LBX-0023-01).

\item   The research of Georgii Riabov is part of the departmental project
  III-05-21 ``Evolution of complicated objects in a random media'' and
  is supported by the Simons Foundation grant aid ``Research during
  the war at the biggest Ukrainian mathematical institution''.
\end{itemize}

\printbibliography{}

\end{document}